\documentclass[12pt]{article}
\usepackage{amssymb,amsmath,amsthm,amsxtra,amsfonts}
\usepackage{geometry}
\usepackage{aliascnt}
\usepackage{color}
\usepackage[usenames,dvipsnames,table]{xcolor}
\usepackage[hypertexnames=false,colorlinks=true,citecolor=Blue,urlcolor=Blue,linkcolor=Blue]{hyperref}
\usepackage{subcaption}
\usepackage{bbm}
\usepackage{dsfont}
\usepackage{graphicx}
\usepackage{mathrsfs}
\usepackage{tikz}
\usetikzlibrary{plotmarks,arrows,calc,patterns}
\usepackage{pgfplots}
\usepackage[longnamesfirst,round]{natbib}
\usepackage{doi}
\usepackage{threeparttable}
\usepackage{booktabs}
\usepackage{setspace}

\theoremstyle{theorem}
\newtheorem{theorem}{Theorem}[section]
\newtheorem{proposition}{Proposition}[section]
\newtheorem{lemma}{Lemma}[section]

\newtheorem{assumption}{Assumption}[section]

\newcommand{\R}{\ensuremath{\mathbf{R}}}

\newcommand{\Nn}{\ensuremath{\mathbf{N}}}

\newcommand{\Gbb}{\ensuremath{\mathbb{G}}}

\newcommand{\E}{\ensuremath{\mathds{E}}}
\newcommand{\one}{\ensuremath{\mathds{1}}}

\newcommand{\dx}{\ensuremath{\mathrm{d}}}
\newcommand{\Var}{\ensuremath{\mathrm{Var}}}

\newcommand{\si}{\perp \! \! \! \perp}
\DeclareMathOperator*{\argmin}{arg\,min}

\newcommand{\Dc}{\ensuremath{\mathcal{D}}}

\newcommand{\Fc}{\ensuremath{\mathcal{F}}}
\newcommand{\Gc}{\ensuremath{\mathcal{G}}}
\newcommand{\Hc}{\ensuremath{\mathcal{H}}}

\newcommand{\Lc}{\ensuremath{\mathcal{L}}}

\newcommand{\Pc}{\ensuremath{\mathcal{P}}}
\newcommand{\Rc}{\ensuremath{\mathcal{R}}}
\newcommand{\Sc}{\ensuremath{\mathcal{S}}}

\begin{document}

\title{\bf Honest Confidence Sets in Nonparametric IV Regression and Other Ill-Posed Models}
\author{Andrii Babii\footnote{University of North Carolina at Chapel Hill - Gardner Hall, CB 3305
		Chapel Hill, NC 27599-3305. Email: \href{mailto:babii.andrii@gmail.com}{babii.andrii@gmail.com}. First draft: March 2016. This paper is a revised first chapter of my Ph.D. thesis. I'm grateful to the associate editor and valuable referees for constructive criticism and suggestions of how to improve the paper.  I'm deeply indebted to my advisor Jean-Pierre Florens and other members of my Ph.D. committee: Eric Gautier, Ingrid Van Keilegom, and Timothy Christensen for helpful suggestions and insightful conversations. This paper also benefited from discussions with Christian Bontemps, Samuele Centorrino, Jasmin Fliegner, Emanuele Guerre, Vitalijs Jascisens, Jihyun Kim, Rohit Kumar, Elia Lapenta, Pascal Lavergne, Thierry Magnac, Andr\'{e} Mas, Nour Meddahi, Markus Reiss, and Shruti Sinha.}
	\\ UNC Chapel Hill
}

\date{October 24, 2019}

\maketitle

\begin{abstract}
	This paper develops inferential methods for a very general class of ill-posed models in econometrics encompassing the nonparametric instrumental variable regression, various functional regressions, and the density deconvolution. We focus on uniform confidence sets for the parameter of interest estimated with Tikhonov regularization, as in \cite{darolles2011nonparametric}. Since it is impossible to have inferential methods based on the central limit theorem, we develop two alternative approaches relying on the concentration inequality and bootstrap approximations. We show that expected diameters and coverage properties of resulting sets have uniform validity over a large class of models, i.e., constructed confidence sets are honest. Monte Carlo experiments illustrate that introduced confidence sets have reasonable width and coverage properties. Using U.S. data, we provide uniform confidence sets for Engel curves for various commodities.
\end{abstract}


\begin{jel}\; C14, C36
\end{jel}

\section{Introduction}
This paper develops honest and uniform confidence sets for the structural function $\varphi$ in a generic class of ill-posed models treated with the Tikhonov regularization. The leading example is a nonparametric instrumental variable regression (NPIV) studied in \cite{florens2003inverse}, \cite{newey2003instrumental},  \cite{darolles2011nonparametric}, \cite{hall2005nonparametric}, and \cite{blundell2007semi}. The NPIV model is
\begin{equation}\label{eq:iv0}
Y = \varphi(Z) + U,\qquad \E[U|W]=0,
\end{equation}
where $Z\in\R^p$ is a vector of explanatory variables, $W\in\R^q$ is a vector of instruments, and $\varphi$ is the function of interest. The NPIV model is ill-posed in the sense that the map from the distribution of the data to the function $\varphi$ is not continuous. As a result, we need to introduce some amount of regularization that would smooth out discontinuities and yield a consistent estimator.

In empirical studies, the function $\varphi$ represents a structural economic relation, such as an Engel curve, cost function, or demand curve. It is not enough to estimate this function to infer possible economic effects. We can only infer the range of possible economic effects from confidence sets. This paper is the first to provide inferential methods for Tikhonov-regularized estimators. I focus on \textit{uniform inference}, which amounts to constructing a set containing the \textit{entire} function $\varphi$ with a high probability. Uniform inference allows one to assess global features of the estimated function and to quantify the range of possible economic effects compatible with the data. Global features may include the evidence for non-linearities, the amount of endogeneity bias compared to the local polynomial estimator, monotonicity, concavity/convexity, or other shape properties. In contrast, pointwise confidence intervals only contain the value $\varphi(z_0)$ at some particular point $z_0$ with high probability and do not provide a valid inference for the \textit{entire} function $\varphi$. Another feature of confidence sets constructed in this paper is \textit{honesty} in the sense of \cite{li1989honest}. Honesty means that coverage properties have uniform validity over a large class of specified models.\footnote{See Eq.~\ref{eq:honest} in Section~\ref{sec:honest} for a formal definition.} Honesty is desirable since the underlying model is never known and coverage properties of dishonest sets may vary from one model to another.

Building uniform confidence sets for a function requires approximating the distribution of the supremum of a certain stochastic process. We show that for a broad class of ill-posed models treated with Tikhonov regularization, the stochastic process driving the distribution of the estimator does not converge weakly in the space of functions. As a result, \textit{it is not possible} to build the uniform confidence set by relying on the uniform central limit theorem which calls for alternative approaches to inference.

To construct confidence sets with good coverage properties, we rely on two different approaches. The first inferential method developed in this paper relies on estimates of tail probabilities with a suitable concentration inequality. A simple illustration of this approach is to build a uniform confidence band for the empirical distribution function using the Dvoretzky-Kiefer-Wolfowitz inequality. Given the empirical distribution function $F_n(x)=\frac{1}{n}\sum_{i=1}^n\one_{\{X_i\leq x\}}$ based on the i.i.d. sample $(X_i)_{i=1}^n$, the Dvoretzky-Kiefer-Wolfowitz tells us that for any \textit{finite sample size $n$} the probability of $F_n$ deviating from the true distribution function $F$ in the supremum norm declines at the exponential rate
\begin{equation*}
	\Pr\left(\|F_n - F\|_\infty > x\right)\leq 2e^{-2nx^2},\qquad \forall x>0.
\end{equation*}
Setting $x=\sqrt{\frac{\log (2/\gamma)}{2n}}$ and $\gamma\in(0,1)$, the inequality becomes
\begin{equation*}
	\Pr\left(\|F_n - F\|_\infty\leq \sqrt{\frac{\log (2/\gamma)}{2n}}\right) \geq 1-\gamma.
\end{equation*}
This allows us to build a uniform confidence band for $F(x)$ with a guaranteed coverage probability $1-\gamma$ taking $F_n(x)\pm\sqrt{\frac{\log (2/\gamma)}{2n}}$. In this simple example, there is no coverage error (the coverage is at least $1-\gamma$ for any finite sample size $n$), and the diameter of the set shrinks at the rate $1/\sqrt{n}$. This inferential approach is different from the alternative asymptotic approach based on the Donsker central limit theorem, which also leads to confidence sets with diameters shrinking at the rate $1/\sqrt{n}$, but having the coverage level $1-\gamma-o(1)$. The coverage error disappears only as sample size goes to infinity. The empirical distribution function is an unbiased estimator, while the majority of estimators of functions are biased. As a result, confidence sets for such estimators, including the ill-posed models considered in this paper, usually have coverage errors.

In this paper, we rely on a concentration inequality valid for more complex statistics than $\|F_n - F\|_\infty$, see \cite{boucheron2013concentration}. In particular, we exploit a data-driven concentration inequality for the supremum of the variance of the estimator to approximate quantiles of the unknown distribution. This approach does not rely on the existence of the supremum of the Gaussian process approximating the supremum of the variance of the estimator. As a result, it is valid for a broad class of data-generating processes and is especially useful in settings where all other approaches to inference fail.

For confidence sets based on the concentration inequality, we characterize non-asymptotic rates of coverage errors and expected diameters explicitly. It is especially important to know both rates, since consistency alone may not be very informative for inference. For instance, the coverage error may decrease at a rate slower than the shrinking of the confidence set, requiring larger sample sizes to achieve good coverage. We show that the bias of the estimator, not the noise coming from the estimation of the operator, drives the coverage errors of our confidence sets. To the best of our knowledge, convergence rates for coverage errors have not been derived for the NPIV model or other ill-posed models considered in this paper.

As an alternative, we also study a more traditional bootstrap inference that relies on a non-trivial application of Gaussian and bootstrap coupling inequalities developed in a seminal series of papers, \cite{chernozhukov2014gaussian}, \cite{chernozhukov2016empirical}. To the best of our knowledge, our application of these coupling results to the nonparametric IV model and Tikhonov estimators is new.

Although the Tikhonov-regularized NPIV is the leading example, the inferential methods developed in this paper are valid for other ill-posed models, including functional regression models and the density deconvolution model. To the best of our knowledge, despite the extensive existing literature on $L_2$ results for the functional regression models, no uniform convergence rates or uniform inferential methods are currently available. Moreover, Tikhonov regularization plays a prominent role in the statistical learning theory. The inferential results obtained in this paper can also be applied in that setting.

\paragraph*{Contribution and related literature.} This paper is the first to develop honest uniform inferential methods in a general and unifying framework, encompassing different ill-posed models treated with Tikhonov regularization. This paper is also the first to provide uniform inference for several ill-posed inverse problems based on the concentration inequality and to provide the non-asymptotic analysis, including convergence rates of coverage errors. Uniform confidence sets are available only for sieve-type estimators of the NPIV model and the density deconvolution model. The literature is also silent about coverage errors of confidence sets. Moreover, it is not known what sorts of restrictions are needed to construct honest confidence sets for ill-posed inverse problems. In this paper, we show that the class of models should involve both the class of functions and the class of operators, which is different from classical direct nonparametric problems where we need only to restrict the function. \cite{horowitz2012uniform} develop uniform confidence bands for sieve NPIV estimator by first constructing pointwise confidence intervals at a finite grid of points and then letting the number of grid points grow at a certain speed to achieve uniform coverage. \cite{chen2018optimal} develop inferential methods for the sieve NPIV estimator without relying on discrete approximations. They focus on uniform inference for a collection of linear and nonlinear functionals using Yurinskii's coupling and obtain uniform confidence bands in the particular case of point evaluation functionals, see also \cite{belloni2015some} for the conditional mean function and \cite{tao2014inference} for general conditional moment restriction models. \cite{lounici2011global} obtain uniform confidence sets for the wavelet deconvolution estimator by relying on Bousquet's version of Talagrand's concentration inequality.

\cite{kato2016uniform} consider a more general case, where the density of the noise is estimated from an auxiliary sample. This paper studies the estimator based on Fourier inversion and builds on the coupling inequalities developed in \cite{chernozhukov2014gaussian} and \cite{chernozhukov2016empirical}.

We show that the confidence sets developed in this paper may enjoy polynomial convergence rates of coverage errors in mildly ill-posed, and some severely ill-posed, cases. An appealing feature of Tikhonov's and more general spectral regularizations is that the estimator is always well-defined in finite samples and converges to the best approximation to the structural parameter under various identification failures, see \cite{florens2011identification} and \cite{babiiflorens2016b}.

The theoretical validity of a uniform confidence set relies on uniform convergence rates of the corresponding estimator. \cite{gagliardini2012tikhonov} obtain uniform convergence rates for the Tikhonov-regularized minimum distance estimator relying on the Sobolev embedding. \cite{chen2018optimal} show that the sieve nonparametric IV estimator can attain optimal asymptotic uniform convergence rates under a somewhat different set of assumptions. \cite{babiiflorens2016b} derive \textit{non-asymptotic} $L_2$ and uniform risk bounds for more general spectrally regularized estimators not tied to a particular estimator of the conditional mean function and show that it is possible to attain polynomial convergence rates in some severely ill-posed cases.\footnote{The minimax-optimal uniform convergence rates in the severely ill-posed case are logarithmically slow, see \cite{chen2018optimal}. However, in the more restricted smoothness class, when the estimated function can be described by the finite number of generalized Fourier coefficients, or more generally if its smoothness matches the ill-posedness of the operator, one can achieve faster polynomial uniform convergence rates.}

Some pointwise inferential results are available for spectral cut-off estimators when the ill-posed operator is known and is not estimated from the data, see \cite{carrasco2011spectral}, \cite{gagliardini2012nonparametric}, \cite{gautier2013nonparametric}, and \cite{florenshorowitzkeielgom2016}. At the same time, \cite{chen2015sieve} provide pointwise inference and bootstrap confidence bands for conditional moment restriction models treated with a sieve approach and nesting the NPIV model as a special case. \cite{cardot2007clt} provide results on the asymptotic normality of linear functionals, which solves the prediction problem for functional regression models. Lastly, \cite{carrasco2013asymptotic} study asymptotic normality of inner products for Tikhonov estimators.

The structure of the paper is as follows. We introduce the notation in the remaining part of this section. Section~\ref{sec:honest} describes the problem of constructing honest uniform confidence sets and introduces two inferential approaches in the special case of the nonparametric IV model. Under a general set of assumptions that are verified later on in each particular application, we establish convergence rates for coverage errors and diameters of constructed sets, uniform over a general set of models. Section~\ref{sec:flir_deconv} extends these results to different functional regression models, and to the density deconvolution. In Section~\ref{sec:mc} we show how to implement confidence sets in practice for the NPIV estimator and explore their finite-sample properties with Monte Carlo experiments. Section~\ref{sec:engel} considers the empirical application to Engel curves, and Section~\ref{sec:conclusion} concludes.

\paragraph{Notation.}
Let $L_2[0,1]^p$ denote the space of functions on some compact set $[0,1]^p\subset \R^p$, square integrable with respect to the Lebesgue measure $\lambda$. For $\varphi\in L_2[0,1]^p$, let $\|.\|$ denote the usual $L_2$ norm derived from the inner product $\langle .,.\rangle$. Let $C[a,b]^p$ denote the space of continuous functions endowed with supremum norm $\|.\|_\infty$. For some positive real number $\beta$, let $C^\beta_M[a,b]^p$ denote the class of $\beta$-H\"{o}lder functions on $(a,b)^p$ with $0<M<\infty$
\begin{equation*}
	C_M^\beta[a,b]^p = \left\{\varphi\in C[a,b]^p:\; \max_{|k|\leq \lfloor\beta\rfloor}\|f^{(k)}\|_\infty \leq M,\; \max_{|k|=\lfloor\beta\rfloor}\sup_{z\ne z'}\frac{\left|\varphi^{(k)}(z) - \varphi^{(k)}(z')\right|}{\|z-z'\|^{\beta-\lfloor\beta\rfloor}}\leq M\right\},
\end{equation*}
where $k=(k_1,\dots,k_p)\in\Nn^p$ is a multi-index, $|k|=\sum_{j=1}^pk_j$, $\varphi^{(k)}(z)=\frac{\partial^{|k|}\varphi(z)}{\partial z_1^{k_1}\dots\partial z_p^{k_p}}$, and $\lfloor \beta\rfloor$ is the largest integer strictly smaller than $\beta$. For $A\subset\R$, let $BV(A)$ be a set of functions of bounded variation on $A$. Let $\Lc_{2}$, $\Lc_{2,\infty}$ and $\Lc_{\infty}$ be spaces of bounded linear operators from $L_2$ to $L_2$, from $L_2$ to $C$, and from $C$ to $C$. Sets on which functions are defined should be clear from the context. Spaces $\Lc_2,\Lc_{2,\infty}$, and $\Lc_{\infty}$ are endowed with standard operator norms, denoted  by $\|K\|=\sup_{\|\varphi\|\leq 1}\|K\varphi\|$, $\|K\|_{2,\infty} = \sup_{\|\varphi\|\leq1}\|K\varphi\|_\infty$, and $\|K\|_\infty=\sup_{\|\varphi\|_\infty\leq1}\|K\varphi\|_\infty$. For $K\in\Lc_{2}$, let $K^*$ denote its Hilbert adjoint operator. Let $\Rc(T)$ and $\Dc(T)$ be the range and the domain of the operator $T$. For two real numbers $a$ and $b$, I denote $a\wedge b = \min\{a,b\}$ and $a\vee b = \max\{a,b\}$.

\section{Honest uniform confidence sets}
We focus on uniform confidence sets, honest to some class of models $\Pc$. The class of $\Pc$ consists of probability distributions of the data satisfying certain restrictions. For a given level $\gamma\in(0,1)$, the \textit{honest} $1-\gamma$ \textit{uniform confidence set}, denoted $C_{n,1-\gamma}=\left\{C_{n,1-\gamma}(z)=\left[C_l(z), C_u(z)\right],\;z\in[0,1]^p\right\}$, satisfies the following coverage
\begin{equation}\label{eq:honest}
\inf_{P\in\Pc}\Pr\left(\varphi(z)\in C_{n,1-\gamma}(z),\;\forall z\in[0,1]^p\right)\geq 1-\gamma - O(\delta_n),
\end{equation}
for some sequence $\delta_n\to 0$. Honesty is desirable for theoretical and practical reasons. On the theoretical side, it is well-known that focusing on a fixed model leads to inconsistent concepts of optimality for non-parametric models; see, e.g., \cite{tsybakov2009introduction}, p.16-19. Honesty also ensures that for a sufficiently large sample size $n$, independent from the distribution of the data, the coverage level of the set will be close to $1-\gamma$, no matter what model in $\Pc$ nature gives us. Honesty is desirable in practice since the underlying probability distribution is unknown. In contrast, a \textit{dishonest} set requires a weaker condition
\begin{equation*}
	\inf_{P\in\Pc}\liminf_{n\to\infty}\Pr\left(\varphi(z)\in C_{n,1-\gamma}(z),\;\forall z\in[0,1]^p\right)\geq 1-\gamma
\end{equation*}
and the sample size $n$ needed to achieve the coverage close to $1-\gamma$ will depend on the distribution of the data.

The second requirement for the confidence set is that its expected diameter under the supremum norm, denoted $|C_{n,1-\gamma}|_\infty=\|C_u-C_l\|_\infty$, shrinks at some rate $\rho_n\to0$
\begin{equation*}
	\sup_{P\in\Pc}\E\left|C_{n,1-\gamma}\right|_\infty = O(\rho_n).
\end{equation*}
We will also consider a variation of this requirement when the diameter is bounded in probability only.

Nonparametric estimation usually involves a bias-variance trade-off for the risk of the estimator. This trade-off translates into a trade-off between the rate $\delta_n$ at which coverage errors tend to zero and the rate $\rho_n$ at which the expected diameter tends to zero. Roughly speaking, we show that the coverage error is driven by the bias of the estimator, while the variance determines its diameter. It is vital to know both rates and requiring only the limiting coverage $1-\gamma$ in the Eq.~\ref{eq:honest} can be misleading, as the coverage error of the set may tend to zero arbitrarily slow. In the following sections, we show that the confidence sets based on the concentration inequality allow us to characterize coverage errors explicitly.

It is also worth mentioning that among confidence sets based on different $L_p,p\in[1,\infty]$ norms, only uniform confidence sets have appealing visualization and are easy to implement numerically. For example, if $\varphi$ is the function on the real line, the confidence set becomes a band on the plane, which contains the whole graph of $\varphi$ with high probability.

\section{Nonparametric IV}\label{sec:honest}
\subsection{The model}
The nonparametric IV model
\begin{equation*}
	Y = \varphi(Z) + U,\qquad \E[U|W]=0
\end{equation*}
leads to the functional equation
\begin{equation}\label{eq:iv}
\E[Y|W=w] = \E[\varphi(Z)|W=w],
\end{equation}
which is an example of an ill-posed inverse problem. Even if we knew the distribution of the data $(Y,Z,W)$, obtaining $\varphi$ requires inverting the conditional expectation operator, which is typically not continuous. \cite{florens2003inverse} and \cite{darolles2011nonparametric} introduce Tikhonov regularization, see \cite{tikhonov1963solution}, to smooth out discontinuities of inversion.

Multiplying Eq.~\ref{eq:iv} by the marginal density function $f_W$
\begin{equation*}
	r(w) \triangleq \E[Y|W=w]f_W(w) = \int_{[0,1]^p}\varphi(z)f_{ZW}(z,w)\dx z \triangleq (T\varphi)(w).
\end{equation*}
Following \cite{darolles2011nonparametric} we focus on the kernel estimator, but other choices are possible. For simplicity, we use the product kernel and equal bandwidth parameters $h_n\to 0$ as $n\to\infty$ for all coordinates
\begin{equation*}
	\begin{aligned}
		\hat r(w) & = \frac{1}{nh_n^q}\sum_{i=1}^nY_iK_w\left(h_n^{-1}(W_i-w)\right) \\
		\hat f_{ZW}(z,w) & = \frac{1}{nh_n^{p+q}}\sum_{i=1}^nK_z\left(h_n^{-1}(Z_i-z)\right)K_w\left(h_n^{-1}(W_i-w)\right) \\
		(\hat T\varphi)(w) & = \int_{[0,1]^p}\varphi(z)\hat f_{ZW}(z,w)\dx z.
	\end{aligned}
\end{equation*}
The Tikhonov-regularized estimator solves the penalized least-squares problem
\begin{equation*}
	\hat\varphi = \argmin_{\phi\in L_2}\left\|\hat T\phi - \hat r\right\|^2 + \alpha_{n}\|\phi\|^2,
\end{equation*}
which admits a closed-form solution
\begin{equation*}
	\hat\varphi = (\alpha_n I + \hat T^*\hat T)^{-1}\hat T^*\hat r,
\end{equation*}
where $\hat T^*$ is the adjoint operator of $\hat T$. The adjoint operator is a solution to $\langle \hat T\varphi,\psi \rangle = \langle \varphi,\hat T^*\psi\rangle$ and is computed by Fubini's theorem
\begin{equation*}
	(\hat T^*\psi)(z) =  \int_{[0,1]^q}\psi(w)\hat f_{ZW}(z,w)\dx w.
\end{equation*}

\subsection{Ill-posedness and regularization bias}\label{sec:reg_bias}
We first introduce the class of functions and operators for which we wish to obtain the honest coverage.
\begin{assumption}\label{as:source_condition1}
	The structural function $\varphi$ and the 1-1 operator $T:C[0,1]^p\to C[0,1]^q$ belong to the following class
	\begin{equation*}
		\Fc = \Fc_{\beta,t,M,C} = \left\{(\varphi,T)\in C^t_M[0,1]^p\times \Lc_{2,\infty}:\; \varphi = (T^*T)^\beta T^*\psi,\;\kappa(\varphi,\psi,T,T^*)\leq C\right\},
	\end{equation*}
	where $\kappa(\varphi,\psi,T,T^*) = \|\varphi\|_\infty\vee \|T^*\|_{2,\infty}\vee\|\psi\|\vee \|T\|^{-1}$ and $0<\beta,t,M,C<\infty$.
\end{assumption}
Source conditions are at the heart of spectral regularization theory, and describe the regularity of the problem by restricting how ill-posed the operator $T$ is, compared to the smoothness of the parameter of interest $\varphi$ (see also \cite{chen2011rate} for its comparability with other assumptions used in the literature). The present source condition is different from the one used to characterize $L_2$ rates, where we would require $\varphi = (T^*T)^\beta\psi$ only, see \cite{carrasco2007linear}. 

We show in the next section that under this source condition, our confidence sets enjoy polynomial convergence rates of coverage errors. In the severely ill-posed case, when singular values of the operator $T$ decay to zero exponentially fast, the source condition in Assumption~\ref{as:source_condition1} requires that the function $\varphi$ is well-approximated by a small number of generalized Fourier coefficients with respect to the SVD basis of $T$. It is a remarkable fact that we can still achieve polynomial convergence rates in this case, since many functions encountered in practice have rapidly declining generalized Fourier coefficients.

Assumption~\ref{as:source_condition1} allows us to control the regularization bias. Note that in the Proposition that follows, the bound is uniform over the source set, and this will be needed to establish honesty of our confidence sets.
\begin{proposition}\label{prop:regularization_bias}
	Suppose that Assumption~\ref{as:source_condition1} is satisfied, then
	\begin{equation*}
		\sup_{(\varphi,T)\in\Fc}\left\|(\alpha_n I + T^*T)^{-1}T^*r - \varphi\right\|_\infty \leq R\alpha_n^{\beta\wedge 1},
	\end{equation*}
	where $R = C^2\left[\beta^\beta(1-\beta)^{1-\beta}\one_{0<\beta<1} + C^{2(\beta-1)}\one_{\beta\geq1}\right]$.
\end{proposition}
In terms of the bias, the regularization takes advantage of the smoothness up to $\beta=1$. This effect is somewhat reminiscent of the saturation of convergence of the kernel density estimator and can be avoided using the iterated or the extrapolated Tikhonov regularization similarly to using higher or infinite order kernels for the kernel density estimator, see \cite{darolles2011nonparametric} and \cite{carrasco2007linear} for more discussion.

\subsection{Honest confidence sets}
Put
\begin{equation*}
	\|\hat \nu_n^\eta\|_\infty = \left\|\frac{1}{n\sqrt{h_n^q}}\sum_{i=1}^n\eta_i(\alpha_n I + T^*T)^{-1}T^*\hat X_{ni}\right\|_\infty,
\end{equation*}
where $\hat X_{ni}(w) = \frac{1}{\sqrt{h_n^q}}\hat U_iK_w\left(h_n^{-1}(W_i - w)\right)$, $\hat U_i = Y_i - \hat\varphi(Z_i)$, and $\eta_i$ are i.i.d. Rademacher random variables. Let $c_{1-\gamma}^*$ be $1-\gamma$ quantile of
\begin{equation*}
	\|\hat V_n^\varepsilon\|_\infty = \left\|\frac{1}{\sqrt{n}}\sum_{i=1}^n\varepsilon_i \hat X_{ni}\right\|_\infty
\end{equation*}
conditionally on the data $\Dc = (Y_i,Z_i,W_i)_{i=1}^n$ and $\varepsilon_i$ are i.i.d. N(0,1).

The class of models $\Pc$ is restricted by the following assumptions.
\begin{assumption}\label{as:data_npiv}
	(i) $(Y_i,Z_i,W_i)_{i=1}^n$ is an i.i.d. sample of $(Y,Z,W)$; (ii) $f_{ZW}\in C^s([0,1]^{p+q}),s>0$ and there exists $\underline f,\bar f$ such that $0<\underline f\leq f_{ZW}\leq \bar f<\infty$; (iii) $K_z$ and $K_w$ are products of the kernel function $k:\R\to\R$ with $k\in C(\R)\cap BV(\R)$ having order $\lfloor s\rfloor\vee\lfloor t\rfloor$ and such that $k\in L_r(\R),r=1,2,3,4$ and $\int|u|^{s\vee t}|k(u)|\dx u<\infty$; (iv) $\E|Y|^4<\infty$, $\E[|U|^2|W]\geq \underline \sigma_2>0$ a.s., and $|U|\leq F$; (v) the integral operator $T:C[0,1]^p\to C[0,1]^q$ is $1$-$1$; (vi) $h_n$ and $\alpha_n$ tend to $0$ polynomially fast and (a) $h_n^tn\log^2 n\to 0$,  $\alpha_n^{2(\beta\wedge 1 + 1)}h_n^qn\log^2 n \to 0$, and $\alpha_nh_n^{2s+q}n\log^2 n\to 0$; (b) $\alpha_n\log^3 n/h_n^p\to0$, $h_n^{s-q/2}\log n/\alpha_n \to 0$, and $\alpha_n^{\beta\wedge 1}\log n/h_n^{q/2}\to 0$; (c) $\alpha_n^2nh_n^{2q}/\log n\to\infty$.
\end{assumption}
Most of these assumptions impose mild smoothness, boundedness, or moment conditions which are plausible in empirical applications. Conditions imposed on the kernel functions are extremely mild and are verified by all kernel functions commonly used in practice. Similar assumptions are usually made in other settings for uniform nonparametric estimation and inference. Assumption~\ref{as:data_npiv} imposes mild regularity conditions on the data generating process. (i) is the standard sampling assumption and may be relaxed to time series data under some weak dependence conditions, e.g., see \cite{babii2016b}; (ii) may rule out common elements between $Z$ and $W$, see discussion in Section \ref{sec:mc} how to deal with common elements; (iii) allows for boundary-corrected kernels, see \cite{darolles2011nonparametric}\footnote{Otherwise, to avoid problems at end-points, we will make a restriction to the interior of $[0,1]^{p+q}$, and all results should be read as uniform over the interior of this set.} for further discussion. Assumption (v) is a completeness condition which may be questionable in empirical applications. Note, however, that in many cases the completeness condition can be relaxed, and it is possible to have valid inferences for various deviations from the completeness condition, see \cite{babiiflorens2016b} for more details.

Assumption~\ref{as:data_npiv} (vi) involves several conditions on tuning parameters and deserves special attention. It imposes implicitly that different components of the model have sufficient regularity with respect to the dimensions $p$ and $q$. To illustrate that these conditions are not contradictory, suppose that $h_n\sim n^{-c_1}$ and $\alpha_n\sim n^{-c_2}$ for some $c_1,c_2>0$ and that $\beta\geq 1$. Then Assumption~\ref{as:data_npiv} reduces to the following 6 conditions: (1) $c_2>(p\vee q/2)c_1$, (2) $4c_2+qc_1>1$, (3) $2c_2+2qc_1<1$, (4) $c_1>1/t$, (5) $c_1(s-q/2)>c_2$, and (6) $c_2+c_1(2s+q)>1$. Conditions (1)-(3) define the set of feasible value for $c_1$ and $c_2$ depending on the dimension parameters $p$ and $q$. Then (4) and (5)-(6) simply require that smoothness parameters $t$ and respectively $s$ are sufficiently large. For instance, when $p=q=1$, we can take, $h_n\sim n^{-1/5}$ and $\alpha_n\sim n^{-1/4}$, and need $t>5$ and $s>7/4$.

The bootstrap and the concentration inequality-based confidence sets are described as
\begin{equation}\label{eq:npiv_ci}
\begin{aligned}
C^*_{n,1-\gamma} & = \left\{[\hat\varphi(z) - q_n^*,\hat\varphi(z)+q_n^*]:\; z\in[0,1] \right\}, \\
C_{n,1-\gamma} & = \left\{[\hat\varphi(z) - \hat q_n,\hat\varphi(z) + \hat q_n]:\; z\in[0,1] \right\}
\end{aligned}
\end{equation}
with
\begin{equation*}
	\begin{aligned}
		q_n^* & = \frac{c_{1-\gamma}^*\|\hat T^*\|_{2,\infty} + \log^{-1}n}{\alpha_n\sqrt{nh_n^q}}, \\
		\hat q_n & = 2\left\|\hat \nu_n^\eta\right\|_\infty +  \frac{3\|\hat T^*\|_{2,\infty}F\|k\|^q\sqrt{2\log(1/\gamma)} + \log^{-1}n}{\alpha_n\sqrt{nh_n^q}}, \\
	\end{aligned}
\end{equation*}
where the operator norm can be computed using Lemma~\ref{lemma:operator_to_kernel} as
\begin{equation*}
	\|\hat T^*\|_{2,\infty} = \left(\sup_{z}\int|\hat f_{ZW}(z,w)|^2\dx w\right)^{1/2}.
\end{equation*}

The next result provides theoretical justification for our honest uniform confidence sets.
\begin{theorem}\label{thm:cs_npiv}
	Suppose that $\Pc$ consists of distributions satisfying Assumptions~\ref{as:source_condition1} and \ref{as:data_npiv}. Then
	\begin{equation*}
		\inf_{P\in\Pc}\Pr\left(\varphi\in C_{n,1-\gamma}\right) \geq 1 - \gamma - O(\delta_n)
	\end{equation*}
	with
	\begin{equation*}
		\begin{aligned}
			\delta_n & = \left(\frac{h_n^\frac{t-q}{2}}{\alpha_n} + \frac{1}{\alpha_n^{1/2}}\left(\sqrt{\frac{\log h_n^{-1}}{nh_n^{p+q}}} + h_n^s\right) + \alpha_n^{\beta\wedge 1}\right)h_n^{q/2}\alpha_nn^{1/2}\log n + \frac{\log n}{\alpha_n}\left(\frac{1}{\sqrt{nh_n^q}} + h_n^s\right). \\
		\end{aligned}
	\end{equation*}
	Moreover, if Assumption~\ref{as:data_npiv} (v)-(vi) is satisfied
	\begin{equation*}
		\liminf_{n\to\infty}\inf_{P\in\Pc}\mathrm{Pr}\left(\varphi\in C_{n,1-\gamma}^*\right) \geq 1 - \gamma.
	\end{equation*}
\end{theorem}

The proof of this result can be found in the Appendix and is a consequence of more general Theorems~\ref{thm:bootstrap} and \ref{thm:main_ci}. In particular, Theorem~\ref{thm:main_ci} derives coverage errors for concentration inequality-based confidence sets, while Theorem~\ref{thm:bootstrap} establishes the asymptotic validity of bootstrap confidence sets for a general class of ill-posed inverse models. Theorem~\ref{thm:bootstrap} is based on the Gaussian and bootstrap coupling inequalities as well as the anti-concentration inequality of \cite{chernozhukov2014gaussian} and \cite{chernozhukov2016empirical}. Note that concentration-based confidence sets allow us to characterize explicitly the coverage error, which has not been known for confidence sets in the nonparametric IV model. Inspection of the proof of Theorem~\ref{thm:bootstrap} reveals that bootstrap confidence sets have additional coverage errors due to the Gaussian and bootstrap approximations.

To illustrate that conditions needed for $\delta_n\to 0$ are not contradictory, suppose that $h_n\sim n^{-c_1}$ and $\alpha_n\sim n^{-c_2}$ for some $c_1,c_2>0$ and that $\beta\geq 1$. Then we need (1) $c_2>pc_1$, (2) $4c_2+qc_1>1$, (3) $2c_2+qc_1<1$, (4) $c_1>1/t$, (5) $c_1s>c_2$, and (6) $c_2+c_1(2s+q)>1$. Note that the set of feasible values of $c_1,c_2$ is larger than that allowed by Assumption~\ref{as:data_npiv} (vi).

The following theorem describes the expected size of our confidence sets.
\begin{theorem}\label{thm:npiv_cs_npiv}
	Suppose that assumptions of Theorem~\ref{thm:cs_npiv} are satisfied. Then
	\begin{equation*}
		\sup_{P\in\Pc}\E_P\left|C_{n,1-\gamma}\right|_\infty = O\left(\frac{1}{\alpha_n\sqrt{nh_n^q}}\right)\quad \text{and}\quad \left|C_{n,1-\gamma}^*\right|_\infty = O_p\left(\frac{1}{\alpha_nh_n^qn^{1/2}}\right),
	\end{equation*}
	uniformly over $P\in\Pc$.
\end{theorem}
Note that we need $\alpha_n^2nh_n^{q}\to \infty$ to ensure that the concentration-based confidence set shrinks and a slightly stronger condition $\alpha_n^2nh_n^{2q}\to \infty$ to ensure that the bootstrap confidence set shrinks, which is also needed to ensure that the coverage error tends to zero.

\section{Functional regressions and density deconvolution}\label{sec:flir_deconv}
\subsection{Functional regressions}
In the functional linear regression, the real dependent variable $Y$ is explained by the continuous-time stochastic process $Z(t),t\in[0,1]^p$. The distinctive feature of this class of models is that it allows handling high-dimensional data without relying on the \textit{sparsity} assumption, which may be restrictive in this setting. There is vast literature on functional regression models in statistics, see \cite{cardot2003testing}, \cite{hall2007methodology}, and references therein. There is also some growing literature in  econometrics, see \cite{florens2015instrumental}, \cite{benatia2017functional}, and \cite{babii2016b}. However, all these papers focus on the estimation in the $L_2$ norm and to the best of my knowledge, there are no currently available uniform inferential methods for functional regression models.

The functional IV regression model is described as
\begin{equation*}
	Y = \int_{[0,1]^p}\varphi(t)Z(t)\dx t + U,\qquad \E[UW(s)] = 0,\qquad \forall s\in[0,1]^q.
\end{equation*}
The slope parameter $\varphi$ measures the strength of the impact of the process $Z$ at different points $t\in[0,1]^p$. If $W=Z$, we obtain the classical functional linear regression model without endogeneity, see, for instance, \cite{hall2007methodology}, while in the IV case of \cite{florens2015instrumental}, $W$ is some functional instrumental variable, uncorrelated with the error term.

The moment restriction leads to the ill-posed equation
\begin{equation*}
	r(s) \triangleq \E[YW(s)] = \int_{[0,1]^p}\varphi(t)\E[Z(t)W(s)]\dx t \triangleq (T\varphi)(s).
\end{equation*}
The slope parameter $\varphi$ is identified when the covariance operator is 1-1, which generalizes the non-singularity condition for the covariance matrix in the finite-dimensional linear regression model.

A variation of this model is studied in \cite{babii2016b}, where the identification is achieved with real-valued instrumental variables through the conditional moment restriction $\E[U|W]=0$. The identifying restriction is the linear completeness condition. Conditional mean-independence leads to the following ill-posed equation
\begin{equation}\label{eq:flir_babii}
\E[Y|W] = \int_{[0,1]^p}\varphi(t)\E[Z(t)|W]\dx t,
\end{equation}
which for appropriate families of functions $\Psi(s,.)$ can equivalently be written as
\begin{equation*}
	r(s) \triangleq \E[Y\Psi(s,W)] = \int_{[0,1]^p}\varphi(t)\E[Z(t)\Psi(s,W)]\dx t\triangleq (T\varphi)(s).
\end{equation*}

In what follows we denote by $W(s)$ either the regressor $Z(s)$, some functional IV $W(s)$, or the function $\Psi(s,W)$ of some real IV $W$. This notation allows encompassing all variations of the functional regression discussed above.

Unlike the NPIV, which requires a non-parametric estimation of the joint density function, all components of functional regression models are estimated at the parametric rate (in the $L_2$ norm) using sample analogs to population moments
\begin{equation*}
	\hat r(s) = \frac{1}{n}\sum_{i=1}^nY_iW_i(s),\qquad \hat k(t,s) = \frac{1}{n}\sum_{i=1}^nZ_i(t)W_i(s).
\end{equation*}
Operator $T$ can be estimated as
\begin{equation*}
	(\hat T\varphi)(s) = \int_{[0,1]^p}\varphi(t)\hat k(t,s)\dx t
\end{equation*}
The adjoint operator $\hat T^*$ can be obtained using Fubini's theorem
\begin{equation*}
	(\hat T^*\varphi)(t) = \int_{[0,1]^q}\psi(s)\hat k(t,s)\dx s.
\end{equation*}

Put
\begin{equation*}
	\|\hat \nu_n^\eta\|_\infty = \left\|\frac{1}{n}\sum_{i=1}^n\eta_i(\alpha_n I + T^*T)^{-1}T^*\hat X_{ni}\right\|_\infty,
\end{equation*}
where $\hat X_{ni}(s) = \hat U_iW_i(s)$, $\hat U_i = Y_i - \langle Z_i,\varphi\rangle$, and $\eta_i$ are i.i.d. Rademacher random variables.

Let $c_{1-\gamma}^*$ be $1-\gamma$ quantile of
\begin{equation*}
	\|\hat V_n^\varepsilon\|_\infty = \left\|\frac{1}{\sqrt{n}}\sum_{i=1}^n\varepsilon_i \hat X_{ni}\right\|_\infty
\end{equation*}
and $\varepsilon_i\sim_{i.i.d.}N(0,1)$.

The following assumption restricts the class of models $\Pc$.
\begin{assumption}\label{as:data_flir}
	(i) $(Y_i,Z_i,W_i)_{i=1}^n$ is an i.i.d. sample of $(Y,Z,W)$; (ii) $Z$ and $W$ have continuous trajectories, $\|Z\|_\infty\leq C$ and $\|UW\|_\infty\leq F$; (iii) $W\in C^s_M[0,1]^q$ and $ZW\in C_M^s[0,1]^{p+q}$, for some $s>(p+q)/2$; (iv) $\E|UW(s)|\geq \underline\sigma>0$ for all $s\in[0,1]^q$; (v) the integral operator $T:C[0,1]^p\to C[0,1]^q$ is 1-1;  (vi) $\alpha_n\to 0$ polynomially fast and $\alpha_nn/\log^2 n\to \infty$ and $\alpha_n^{2(\beta\wedge 1) + 2}n\log^2n\to0$.
\end{assumption}
Note that the functional linear regressions admit a simpler characterization and involve even less stringent assumptions than the nonparametric IV model.

The bootstrap and the concentration inequality-based confidence sets are
\begin{equation*}
	C^*_{n,1-\gamma} = \left\{[\hat\varphi(t) - q_n^*,\hat\varphi(t)+q_n^*]:\; t\in[0,1] \right\}
\end{equation*}
and
\begin{equation*}
	C_{n,1-\gamma} = \left\{[\hat\varphi(t) - \hat q_n,\hat\varphi(t) + \hat q_n]:\; t\in[0,1] \right\}
\end{equation*}
with
\begin{equation*}
	\begin{aligned}
		q_n^* & = \frac{c_{1-\gamma}^*\|\hat T^*\|_{2,\infty} + \log^{-1}n}{\alpha_n\sqrt{n}}, \\
		\hat q_n & = 2\left\|\hat \nu_n^\eta\right\|_\infty +  \frac{3\|\hat T^*\|_{2,\infty}F\sqrt{2\log(1/\gamma)} + \log^{-1}n}{\alpha_n\sqrt{n}}, \\		
	\end{aligned}
\end{equation*}
where the operator norm can be computed using Lemma~\ref{lemma:operator_to_kernel} as
\begin{equation*}
	\|\hat T^*\|_{2,\infty} = \left(\sup_{t}\int|\hat k(t,s)|^2\dx s\right)^{1/2}.
\end{equation*}

The next result provides theoretical justification for confidence sets based on the concentration inequality and bootstrap confidence sets.
\begin{theorem}\label{thm:confidence_set_flir}
	Suppose that Assumptions~\ref{as:source_condition1} and \ref{as:data_flir} are satisfied. Then for any $\gamma\in(0,1)$
	\begin{equation*}
		\inf_{P\in\Pc}\mathrm{Pr}\left(\varphi\in C_{n,1-\gamma}\right) \geq 1 - \gamma - O\left(\left(\alpha_n^{\beta\wedge 1+1}n^{1/2} + \frac{1}{\sqrt{\alpha_nn}}\right)\log n\right).
	\end{equation*}
	Moreover, if Assumption~\ref{as:data_decon} (vi) is satisfied
	\begin{equation*}
		\liminf_{n\to\infty}\inf_{P\in\Pc}\mathrm{Pr}\left(\varphi\in C_{n,1-\gamma}^*\right) \geq 1 - \gamma.
	\end{equation*}
\end{theorem}

Note that for concentration inequality-based confidence sets the coverage error is effectively driven by the bias induced by the regularization and kernel smoothing. For bootstrap confidence sets we have an additional coverage error due to the Gaussian and bootstrap approximations and need additional condition $nh_n^2/\log^7n\to \infty$. The following result describes the expected size of various confidence sets.
\begin{theorem}\label{thm:flir_cs_rate}
	Suppose that assumptions of Theorem~\ref{thm:confidence_set_deconv} are satisfied, then
	\begin{equation*}
		\sup_{P\in\Pc}\E_P\left|C_{n,1-\gamma}\right|_\infty = O\left(\frac{1}{\alpha_nn^{1/2}}\right)\qquad \text{and}\qquad  \left|C_{n,1-\gamma}^*\right|_\infty = O_p\left(\frac{1}{\alpha_nn^{1/2}}\right),
	\end{equation*}
	uniformly over $P\in\Pc$.
\end{theorem}
To ensure that confidence sets shrink we need $\alpha_n^2n \to \infty$. In both cases, sets shrink at the same speed.

\subsection{Density deconvolution}
Often economic data are not measured precisely. Density deconvolution allows estimating the density of unobserved data from the data measured with errors. Density deconvolution is encountered in a variety of econometric applications, e.g. to the earning dynamics in \cite{bonhomme2010generalized}, to panel data in \cite{evdokimov2010identification}, or to the instrumental regression in \cite{adusumilli2018nonparametric}. In the simplest example of this model, see \cite{carrasco2011spectral}, we have some noisy scalar observations $Y$, of the latent variable $Z$, contaminated by measurement errors $U$
\begin{equation*}
	Y = Z + U,\qquad Z\si U.
\end{equation*}
Distributions of both $Z$ and $U$ are assumed to be absolutely continuous with respect to the Lebesgue measure with densities $\varphi$ and $f$. The density of measurement errors $f$ is assumed to be known. The goal is to recover the density of the latent variable $Z$ from observing contaminated i.i.d. sample $(Y_i)_{i=1}^n$.

Independence and additivity of the noise imply that the density function $r$ of $Y$ satisfies the following convolution equation
\begin{equation}\label{eq:convolution}
r(y) = \int \varphi(z) f(y-z) \dx z =: (T\varphi)(y),
\end{equation}
where the operator $T:L_2\to L_2$. For simplicity of presentation and tractability of our results, let us assume that all densities are continuous, bounded and compactly supported, with support contained inside the interval $[0,1] \subset \R$. This assumption is not very restrictive, since most of the economic variables, e.g., reported earning or costs are bounded.

The adjoint operator to $T$ can be computed using Fubini's theorem
\begin{equation*}
	(T^*\psi)(z) = \int \psi(y)f(y-z) \dx y.
\end{equation*}
We estimate the density of $Y$ from the sample $(Y_i)_{i=1}^n$ using the kernel density estimator
\begin{equation*}
	\hat r(y) = \frac{1}{nh_n}\sum_{i=1}^nK\left(\frac{Y_i - y}{h_n}\right).
\end{equation*}

Put
\begin{equation*}
	\|\hat \nu_n^\eta\|_\infty = \left\|\frac{1}{n\sqrt{h_n}}\sum_{i=1}^n\eta_i(\alpha_n I + T^*T)^{-1}T^*\hat X_{ni}\right\|_\infty,
\end{equation*}
where $\hat X_{ni}(y) = \frac{1}{\sqrt{h_n}}K\left(\frac{Y_i - y}{h_n}\right) - \frac{1}{n\sqrt{h_n}}\sum_{i=1}^nK\left(\frac{Y_i - y}{h_n}\right)$ and $\eta_i$ are i.i.d. Rademacher random variables.

The class of models $\Pc$ is restricted by the following assumptions.
\begin{assumption}\label{as:data_decon}
	(i) $(Y_i)_{i=1}^n$ is an i.i.d. sample of $Y$; (ii) $f,\varphi, r$ are continuous, compactly supported on some subsets of $[0,1]$, and the density $r$ is uniformly bounded away from zero and infinity by some finite constant; (iii) $r\in C_M^s[0,1]$, for some $s>0$; (iv) $K$ is a symmetric continuous square-integrable kernel function of bounded variation of order $\lfloor s\rfloor$ satisfying $\int|u|^s|K(u)|\dx u<\infty$; (v) the integral operator $T:C[0,1]^p\to C[0,1]^q$ is 1-1; (vi) $h_n\to 0$ and $\log h_n^{-1}=O(\log n)$, $nh_n^2/\log^7n\to \infty$, $nh_n^{2s+1}\log^2 n\to 0$, and $\alpha_n^{2(\beta\wedge 1) + 2}nh_n\log^2n\to0$.
\end{assumption}

Note that for this model, operators $T$ and $T^*$ are known, which simplifies the analysis. Let $c_{1-\gamma}^*$ be $1-\gamma$ quantile of
\begin{equation*}
	\|\hat V_n^\varepsilon\|_\infty = \left\|\frac{1}{\sqrt{n}}\sum_{i=1}^n\varepsilon_i \hat X_{ni}\right\|_\infty
\end{equation*}
and $\varepsilon_i\sim_{i.i.d.}N(0,1)$.

The bootstrap and the concentration inequality-based confidence sets are described as
\begin{equation*}
	C^*_{n,1-\gamma} = \left\{[\hat\varphi(z) - q_n^*,\hat\varphi(z)+q_n^*]:\; z\in[0,1] \right\}
\end{equation*}
and
\begin{equation*}
	C_{n,1-\gamma} = \left\{[\hat\varphi(z) - \hat q_n,\hat\varphi(z) + \hat q_n]:\; z\in[0,1] \right\}
\end{equation*}
with
\begin{equation*}
	\begin{aligned}
		q_n^* & = \frac{c_{1-\gamma}^*\|T^*\|_{2,\infty} + \log^{-1}n}{\alpha_n\sqrt{nh_n}}, \\
		\hat q_n & = 2\left\|\hat \nu_{n}^\eta\right\|_\infty +  \frac{6\|T^*\|_{2,\infty}\|K\|\sqrt{2\log(1/\gamma)} + \log^{-1}n}{\alpha_n\sqrt{nh_n}}, \\		
	\end{aligned}
\end{equation*}
where the operator norm can be computed using Lemma~\ref{lemma:operator_to_kernel} as
\begin{equation*}
	\|T^*\|_{2,\infty} = \left(\sup_{z}\int|f(y-z)|^2\dx y\right)^{1/2}.
\end{equation*}

The next result provides theoretical justification for confidence sets based on the concentration inequality and bootstrap confidence sets.
\begin{theorem}\label{thm:confidence_set_deconv}
	Suppose that Assumptions~\ref{as:source_condition1} and \ref{as:data_decon} are satisfied. Then for any $\gamma\in(0,1)$
	\begin{equation*}
		\inf_{P\in\Pc}\mathrm{Pr}\left(\varphi\in C_{n,1-\gamma}\right) \geq 1 - \gamma - O\left(\left(h_n^{s} + \alpha_n^{\beta\wedge 1+1}\right)\sqrt{nh_n}\log n + \frac{\log n}{n^{1/2}}\right).
	\end{equation*}
	Moreover, if Assumption~\ref{as:data_decon} (vi)-(vii) is satisfied
	\begin{equation*}
		\liminf_{n\to\infty}\inf_{P\in\Pc}\mathrm{Pr}\left(\varphi\in C_{n,1-\gamma}^*\right) \geq 1 - \gamma.
	\end{equation*}
\end{theorem}

Note that for concentration inequality-based confidence sets the coverage error is effectively driven by the bias induced by the regularization and kernel smoothing. For bootstrap confidence sets we have an additional coverage error due to the Gaussian and bootstrap approximations and need additional condition $nh_n^2/\log^7n\to \infty$. The following result describes the expected size of various confidence sets.
\begin{theorem}\label{thm:deconv_cs_rate}
	Suppose that assumptions of Theorem~\ref{thm:confidence_set_deconv} are satisfied, then
	\begin{equation*}
		\sup_{P\in\Pc}\E_P\left|C_{n,1-\gamma}\right|_\infty = O\left(\frac{1}{\alpha_n\sqrt{nh_n}}\right)\qquad \text{and}\qquad \left|C_{n,1-\gamma}^*\right|_\infty = O_p\left(\frac{1}{\alpha_nh_nn^{1/2}}\right),
	\end{equation*}
	uniformly over $P\in\Pc$.
\end{theorem}
To ensure that a concentration inequality-based confidence set shrinks, we need $\alpha_n^2h_nn \to \infty$ while for the bootstrap confidence set we need the slightly stronger requirement $\alpha_n^2h_n^2n\to \infty$.

\section{Implementation and Monte Carlo experiments}\label{sec:mc}
This section reports results of Monte Carlo experiments for our confidence sets. We focus on the NPIV estimator. Samples of size $n\in\{1000,5000\}$ are generated as follows
\begin{equation*}
	\begin{aligned}
		Y & = \varphi(Z) + U,\qquad \varphi(z) = e^{-z^2/0.8}, \\
		\begin{pmatrix}
			Z \\
			W \\
			U
		\end{pmatrix} & \sim N\left(\begin{pmatrix}
			0 \\
			0 \\
			0
		\end{pmatrix},\begin{pmatrix}
			\sigma^2_z & \rho \sigma_z\sigma_w & \sigma_{zu} \\
			\rho \sigma_z\sigma_w & \sigma^2_w & 0 \\
			\sigma_{zu} & 0 & \sigma^2_u \\
		\end{pmatrix}\right),
	\end{aligned}
\end{equation*}
where $\sigma_z = \sigma_w = 0.3$, $\sigma_u = \sqrt{0.03}, \sigma_{zu} = 0.04$, and $\rho=0.3$. To be consistent with the theory, we keep only observations inside a sufficiently large compact set. The strength of the instrument as measured by the correlation coefficient $\rho$ is calibrated to our empirical application.

Since the joint density of $(Z,W)$ is bivariate normal, this choice of functional forms resembles up to a constant the convolution of Gaussian densities. Therefore, with appropriate modifications, we can easily adapt the simulation design not only to functional regression but also to the deconvolution model. Since the underlying integral equation would be the same in all cases, to economize on space, we do not report MC experiments for other models. We also remark that the normal density leads to rapidly declining singular values. This design corresponds to the most difficult severely ill-posed setting. We consider $5000$ replications of each experiment. 

We consider confidence sets described in Eq.~\ref{eq:npiv_ci}, where the operator norm $\|\hat T^*\|_{2,\infty}$ is computed using the mixed norm of the kernel function of $\hat T^*$ by Lemma~\ref{lemma:operator_to_kernel}. The estimate of the density function $f_{ZW}$ is obtained using kernel smoothing. For simplicity of implementation, we do not optimize the performance with higher-order boundary kernels and take the product of second-order Epanechnikov kernels.

The estimator is discretized using a simple Riemann sum on the grid of $100$ equidistant points. This suffices to ensure that the numerical errors are negligible compared to the statistical noise in our setting. A higher number of grid points or cubature can give better approximation if needed. The discretized estimator has closed-form expression, and its computation boils down to solving a system of linear equations 
\begin{equation*}
	\hat{\boldsymbol{\varphi}} = \left(\alpha_n \mathbf{I} + \mathbf{K}^\top\mathbf{K}\right)^{-1}\mathbf{K}^\top \mathbf{r},
\end{equation*}
where $\mathbf{I}$ is the $T\times T$ identity matrix, $\mathbf{r} = \left(\frac{1}{n}\sum_{i=1}^nY_ih_n^{-1}K\left(h_n^{-1}(W_i - w)\right)\right)_{1\leq j\leq T}$,  $\mathbf{K}=\boldsymbol{f}\Delta$ with $\boldsymbol{f} = (\hat f_{ZW}(z_k,w_j))_{1\leq j,k\leq T}$, and $\Delta$ is the grid step. Therefore, the estimator is straightforward to implement and fast to compute. 

\paragraph{Data-driven choice of tuning parameters.} As discussed in Section~\ref{sec:honest}, for confidence sets, the bias-variance trade-off for the risk of the estimator reduces to the trade-off between coverage errors and the diameter of confidence sets. For larger values of tuning parameters, the band becomes too narrow, and the bias starts to dominate, reducing uniform coverage. On the other side, even though smaller values of tuning parameters reduce bias and the estimator becomes closer to its population value, they also increase the size of the variance and lead to wider bands. The optimal choice of tuning parameters balances the two.

Existing data-driven rules for Tikhonov regularization aim to obtain an accurate estimator, see \cite{feve2010practice} or \cite{centorrino2014data}, and do not take into account the trade-off between the coverage and the diameter. Note that the risk-optimal estimation conflicts with optimal inference. In particular, our theory requires undersmoothing in order to obtain valid inferences. We suggest the following procedure inspired by \cite{bissantz2007non}:
\begin{enumerate}
	\item Use some over-smoothing data-driven rule to select pilot tuning parameters. We consider the cross-validation studied in \cite{centorrino2014data} which consists of solving the leave-one-out empirical counterpart to $\|\hat r - \hat T\hat\varphi_{\alpha,h}\|$
	\begin{equation*}
		(\hat\alpha_J,\hat h_J) = \argmin_{\alpha,h>0}\frac{1}{n}\sum_{i=1}^n\left\{\frac{1}{n-1}\sum_{j\ne i}(Y_j - [K_h\ast\hat\varphi_{\alpha,h}](Z_j))\frac{1}{h}K\left(h^{-1}(W_j - W_i)\right) \right\}^2,
	\end{equation*}
	where $[K_h\ast g](x) = \frac{1}{h}\int g(z)K(h^{-1}(x-z))\dx x$ denotes the convolution operator. To ensure that tuning parameters are over-smoothing, we scale pilot tuning parameters $(\kappa_a\hat\alpha_J,\kappa_h\hat h_J)$ with some $\kappa_a,\kappa_h\geq 1$. 
	\item Compute $\hat\varphi_{h_j,\alpha_j}$ on a grid of points $(j\kappa_a\alpha_J/J,j\kappa_hh_J/J)_{j=1}^J$. Choose the largest $(\alpha_j,h_j)$ such that
	\begin{equation*}
		\|\hat\varphi_{\alpha_j,h_j} - \hat\varphi_{\alpha_{j-1},h_{j-1}}\|_\infty > \tau\|\hat\varphi_{\alpha_J,h_J} - \hat\varphi_{\alpha_{J-1},h_{J-1}}\|_\infty 
	\end{equation*}
\end{enumerate}
Like \cite{bissantz2007non} we find that $J=20$ and $\tau = 2$ work well in practice. We also find that $(\kappa_a,\kappa_h) = (3.5,1.5)$ is suitable for concentration-based confidence sets and $(\kappa_a,\kappa_h) = (1,1.5)$ works well for bootstrap confidence sets. Note that \cite{bissantz2007non} have a single tuning parameter, and they find that the factor $1.5$ is suitable, which is what we find for the bandwidth parameter.

\paragraph{Simulation results.} In Figure~\ref{fig:1_1} we plot estimates with a $95\%$ confidence band based on bootstrap approximation, averaged over $5000$ Monte Carlo experiments. Figure \ref{fig:1_2} represents the same plot for confidence bands based on concentration inequality. We also report MC coverage probabilities $\hat\gamma$. In our Monte Carlo experiments, bootstrap approximation leads to slightly better performance. Resulting confidence bands are narrower, and the estimator is centered more closely to the population value of the parameter of interest. Nonetheless, confidence sets based on the concentration inequality are also informative about global shape properties of the estimated function. For example, Figure~\ref{fig:1_1} (b) gives an excellent confidence band with a reasonably small amount of bias.

\begin{figure}
	\begin{subfigure}[b]{0.5\textwidth}
		\includegraphics[width=\textwidth]{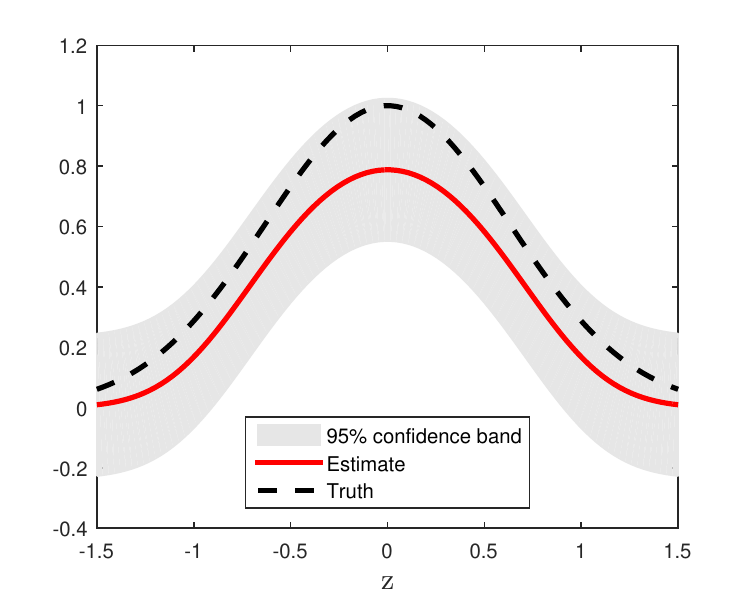}
		\caption{$n = 1000,\hat\gamma = 0.87$}
	\end{subfigure}
	\begin{subfigure}[b]{0.5\textwidth}
		\includegraphics[width=\textwidth]{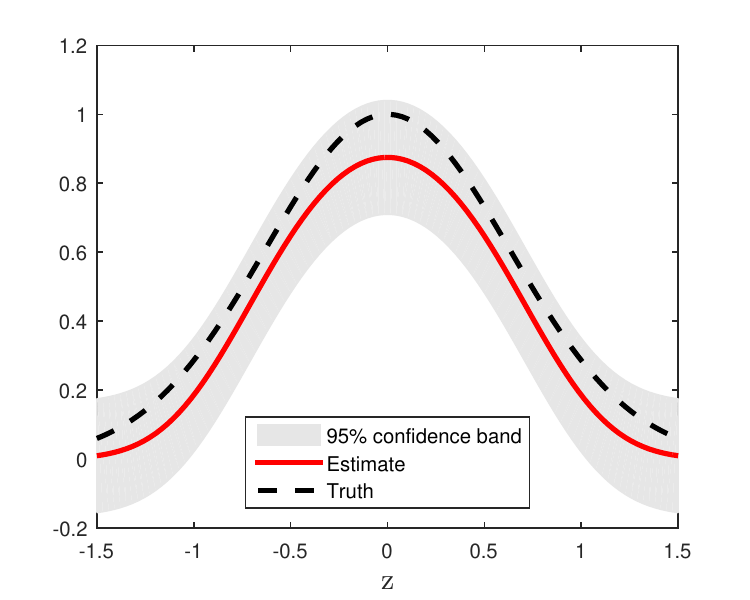}
		\caption{$n = 5000,\hat\gamma = 0.97$}
	\end{subfigure}
	\caption{Bootstrap uniform confidence bands, averaged over $5000$ experiments. $\hat\gamma$ is the MC coverage probability.}
	\label{fig:1_1}
\end{figure}

\begin{figure}
	\begin{subfigure}[b]{0.5\textwidth}
		\includegraphics[width=\textwidth]{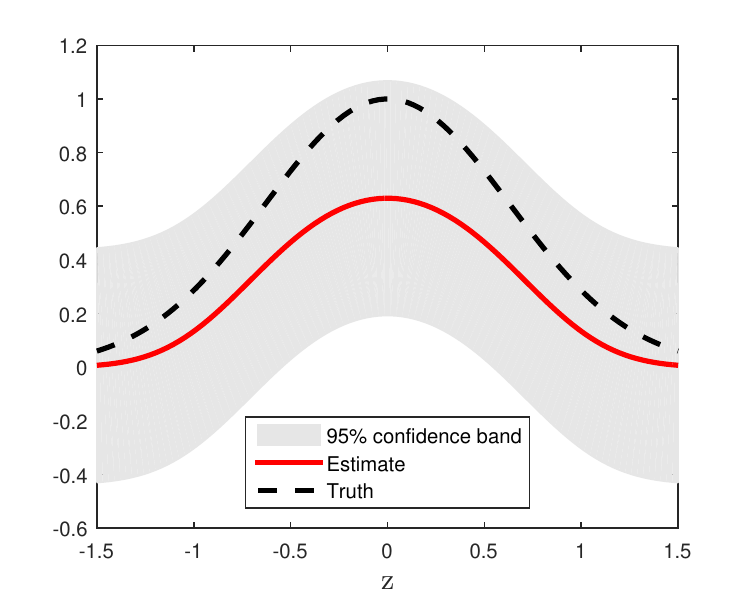}
		\caption{$n = 1000,\hat\gamma = 0.89$}
	\end{subfigure}
	\begin{subfigure}[b]{0.5\textwidth}
		\includegraphics[width=\textwidth]{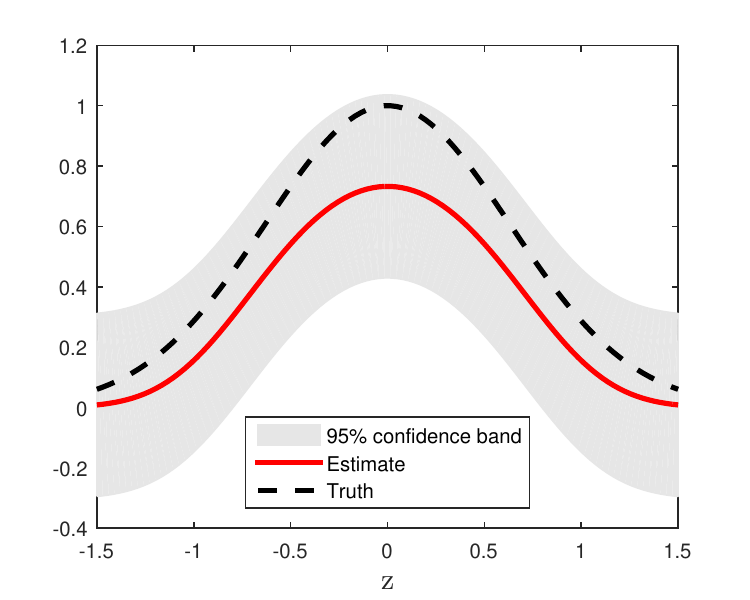}
		\caption{$n = 5000,\hat\gamma = 0.94$}
	\end{subfigure}
	\caption{Concentration-based uniform confidence bands, average over $5000$ experiments. $\hat\gamma$ is the MC coverage probability.}
	\label{fig:1_2}
\end{figure}

\paragraph{Exogenous covariates.}
In the presence of additional exogenous covariates, one can consider several avenues. As an example, suppose that one is interested in estimating the demand function, so $Z=(Z_1,Z_2)$ with $Z_1$ denoting the endogenous price and $Z_2$ denoting exogenous covariates, such as income and other control variables. One way to accommodate the exogenous covariates is to estimate the semiparametric partially linear specification allowing for non-linearities in prices, see \cite{yatchew2001household} and \cite{florens2015instrumental}.

Alternatively, one may want to keep the fully non-parametric specification. In this case, the model $Y=\varphi(Z_1,Z_2)+U,\E[U|Z_2,W]=0$ leads to the following family of ill-posed integral equations
\begin{equation*}
	\E[Y|Z_2=z_2,W=w]f_{Z_2,W}(z_2,w) = \int\varphi(z_1,z_2)f_{Z_1,Z_2,W}(z_1,z_2,w)\dx z_1.
\end{equation*}
For any fixed value of $z_2$, we are back to our setting, and our uniform inferential methods would allow us to construct uniform confidence sets for $z_1\mapsto \varphi(z_1,z_2)$ for any fixed value of $z_2$. Uniform inference for the demand function for different income levels might be sufficient to reveal the amount of non-linearities and/or endogeneity. Constructing the joint uniform confidence set for the bivariate function $\varphi$ and the theoretical analysis of such sets is beyond the scope of the present paper.

\section{Engel curves in the United States}\label{sec:engel}
In this section we estimate Engel curves using the NPIV approach and construct uniform confidence sets using the bootstrap. Engel curves describe how the demand for a commodity changes while the household's budget increases. Estimation of Engel curves is fundamental for the analysis of consumer behavior and has implications in different fields of empirical research. Interesting applications include the measurement of welfare losses associated with tax distortions in \cite{banks1997quadratic}, estimation of growth and inflation in \cite{nakamura2014chinese}, or estimation of income inequality across countries in \cite{almaas2012international}.

Previously \cite{blundell2007semi} estimated a shape-invariant system of Engel curves for food, alcohol, and fuel on UK data with a sieve approach. For simplicity, we focus on the non-parametric specification of the Engel curve.

Our dataset is drawn from the 2015 US Consumer Expenditure Survey data, and we estimate Engel curves for a set of goods, including food, tobacco, alcohol, gas and oil, and health. In our subsample, we have married couples with positive income during the past 12 months, including households with and without children. The dependent variable is the share of expenditures on the particular commodity. The log of total expenditures on non-durable goods is used as an independent variable. As in \cite{blundell2007semi} we instrument the log of total expenditures with income before tax. However, in contrast to \cite{blundell2007semi}, we also estimate Engel curves for tobacco and health care.

We plot NPIV estimates with a $95\%$ uniform confidence bands in Fig.~\ref{fig:3}. We also plot the local linear estimator, which does not correct for the endogeneity bias. Estimates in Figure~\ref{fig:3} illustrate that there is significant endogeneity bias in the Engel curves. Indeed, the local polynomial estimator is often outside the confidence band. For instance, the local polynomial estimator largely overestimates the shape of the Engel curve for higher budgets. In most of the cases the estimated Engel curves exhibit more curvature than suggested by the local polynomial estimator. While it may look like Engel curves slightly increase at the beginning, this observation is not statistically significant and most likely can be attributed to boundary effects.

\begin{figure}	
	\caption{Engel curves estimated with NPIV (solid blue line) with 95\% bootstrap confidence band and local polynomial estimator (dashed red line). Horizontal axis: total expenditures. Vertical axis: budget share.}
	\begin{subfigure}[b]{0.43\textwidth}
		\includegraphics[width=\textwidth]{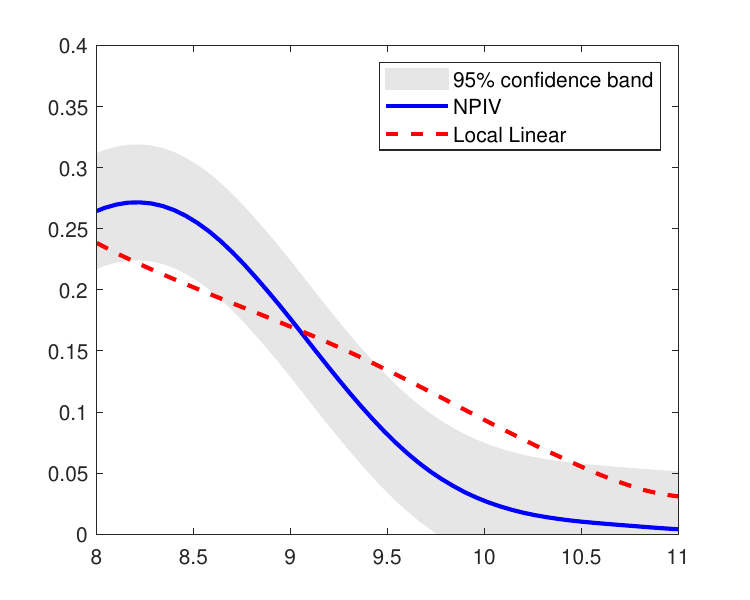}
		\caption{Food home}
	\end{subfigure}
	\begin{subfigure}[b]{0.43\textwidth}
		\includegraphics[width=\textwidth]{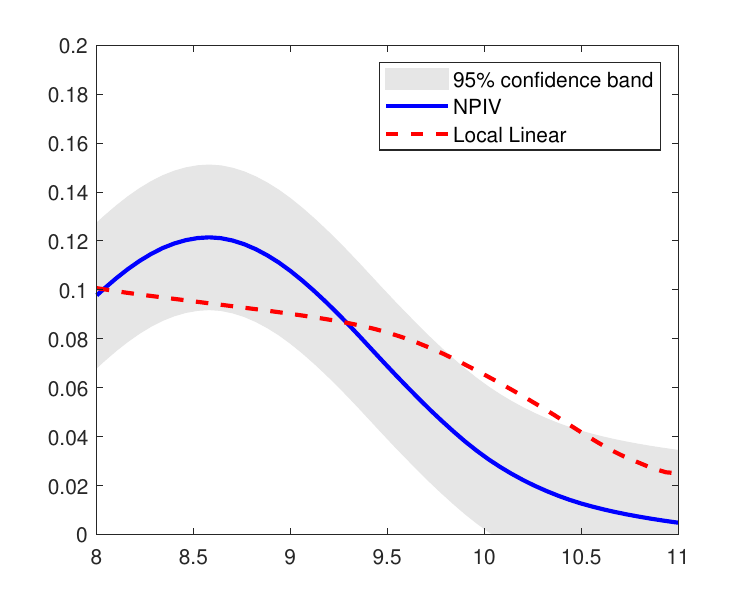}
		\caption{Food away}
	\end{subfigure}	
	\begin{subfigure}[b]{0.43\textwidth}
		\includegraphics[width=\textwidth]{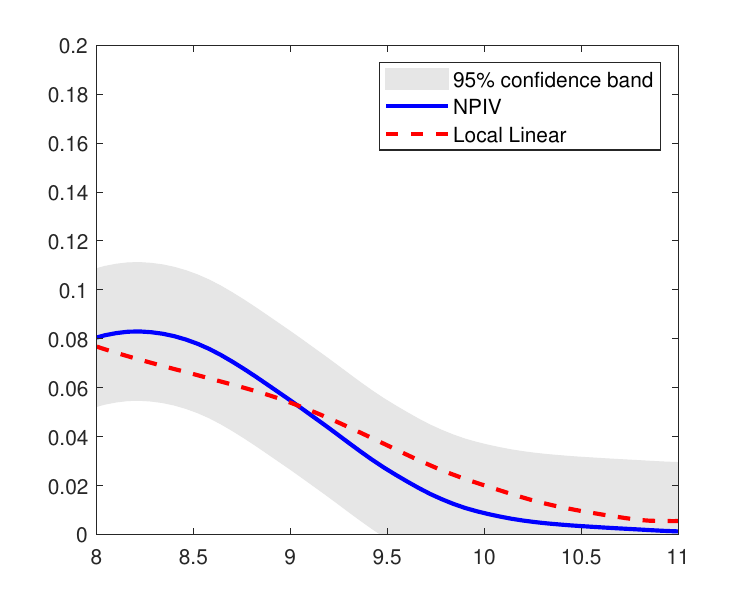}
		\caption{Tobacco}
	\end{subfigure}
	\begin{subfigure}[b]{0.43\textwidth}
		\includegraphics[width=\textwidth]{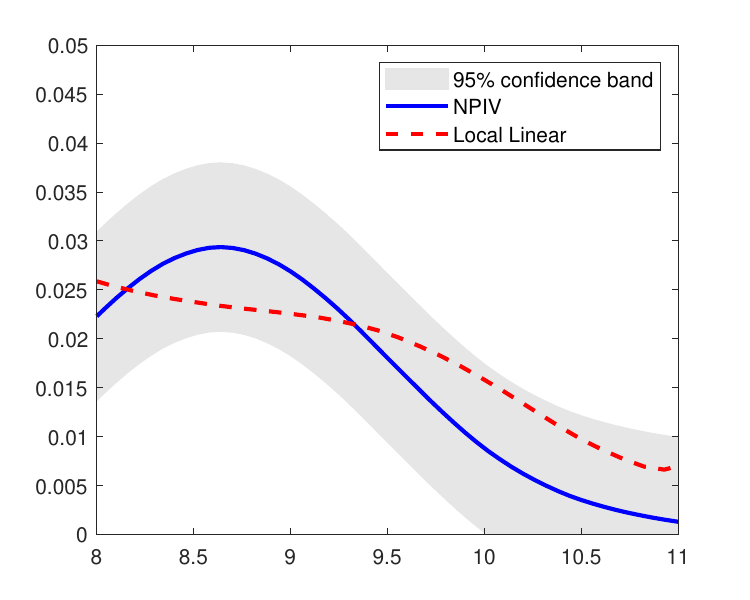}
		\caption{Alcohol}
	\end{subfigure}
	\begin{subfigure}[b]{0.43\textwidth}
		\includegraphics[width=\textwidth]{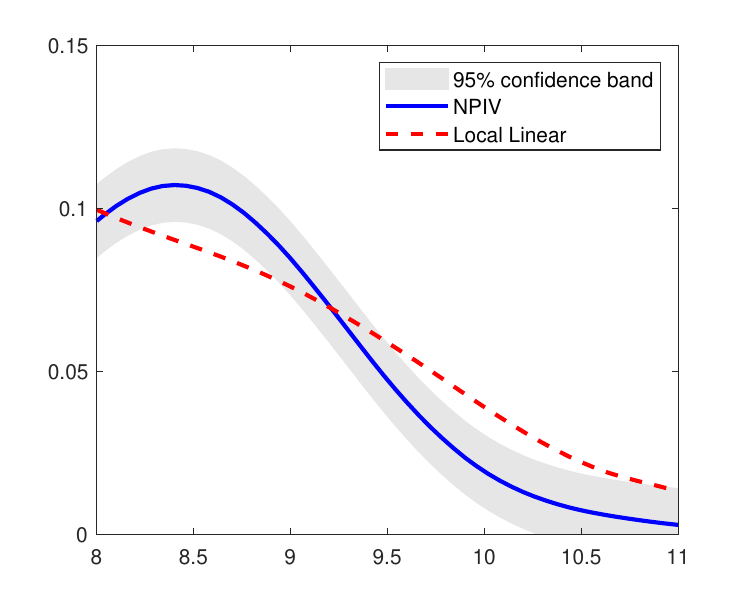}
		\caption{Gas and oil}
	\end{subfigure}\qquad\qquad\quad
	\begin{subfigure}[b]{0.43\textwidth}
		\includegraphics[width=\textwidth]{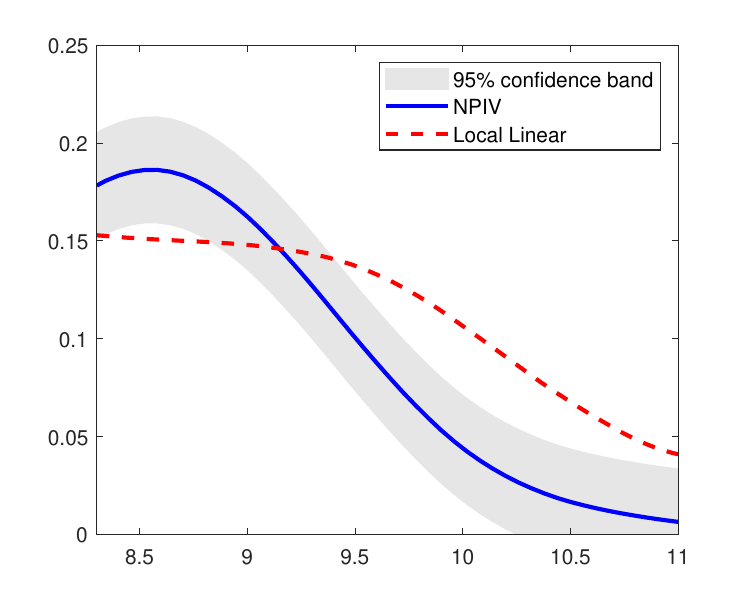}
		\caption{Health}
	\end{subfigure}
	\label{fig:3}
\end{figure}

\section{Conclusions}\label{sec:conclusion}
This paper studies uniform inferential methods in ill-posed models treated with Tikhonov regularization. Building uniform confidence sets in this setting is a difficult problem and requires approximating the distribution of the supremum of a complex empirical process. We show that it is not possible to establish functional convergence of this process for a very general class of ill-posed models known in econometrics and statistics. Nonetheless, we demonstrate that it is possible to obtain uniform inference relying on alternative methods.

We develop two approaches to uniform inference that lead to honest confidence sets. Honest confidence sets are of practical interest for several reasons. They ensure the existence of a certain sample size after which the coverage level will be not significantly smaller than the nominal coverage level, regardless of how complex the estimated function is within a given smoothness class. Moreover, it is widely recognized in non-parametric statistics, results for a fixed model can lead to inconsistent notions of optimality. Honest confidence sets, on the other hand, have the uniform validity.

The first approach developed in this paper relies on data-driven concentration inequality and allows for non-asymptotic characterization of coverage errors and expected diameters. These rates are new and have not been previously discussed in the literature. We also develop the asymptotic bootstrap approach to inference relying on the coupling inequalities of \cite{chernozhukov2016empirical} and show that the resulting confidence sets have desirable theoretical properties.

Both methods demonstrate good performance in Monte Carlo experiments, and it seems that bootstrap confidence sets are slightly more narrow for a comparable coverage level. On the other hand, concentration-based confidence sets are computationally less demanding and still reveal lots of information about the global shape properties of functional parameters. Lastly, unlike bootstrap and/or Gaussian approximation-based inference, concentration inequalities do not impose any restrictions on indexing classes of empirical processes. Therefore, this approach can be easily adapted to other complex nonparametric estimators.

\newpage

\bibliography{references}

\begin{thebibliography}{54}
\providecommand{\natexlab}[1]{#1}
\providecommand{\url}[1]{\texttt{#1}}
\expandafter\ifx\csname urlstyle\endcsname\relax
  \providecommand{\doi}[1]{doi: #1}\else
  \providecommand{\doi}{doi: \begingroup \urlstyle{rm}\Url}\fi

\bibitem[Adusumilli and Otsu(2018)]{adusumilli2018nonparametric}
K.~Adusumilli and T.~Otsu.
\newblock Nonparametric instrumental regression with errors in variables.
\newblock \emph{Econometric Theory}, 34\penalty0 (6):\penalty0 1256--1280,
  2018.
\newblock \doi{10.1017/S0266466617000469}.

\bibitem[Almås(2012)]{almaas2012international}
I.~Almås.
\newblock International income inequality: Measuring ppp bias by estimating
  engel curves for food.
\newblock \emph{American Economic Review}, 102\penalty0 (2):\penalty0
  1093--1117, 2012.
\newblock \doi{10.1257/aer.102.2.1093}.

\bibitem[Babii(2020)]{babii2016b}
A.~Babii.
\newblock High-dimensional mixed-frequency iv regression.
\newblock \emph{arXiv:2003.13478}, 2020.

\bibitem[Babii and Florens(2020)]{babiiflorens2016b}
A.~Babii and J.-P. Florens.
\newblock Is completeness necessary? estimation in nonidentified linear models.
\newblock \emph{arXiv:1709.03473}, 2020.

\bibitem[Banks et~al.(1997)Banks, Blundell, and Lewbel]{banks1997quadratic}
J.~Banks, R.~Blundell, and A.~Lewbel.
\newblock Quadratic engel curves and consumer demand.
\newblock \emph{Review of Economics and Statistics}, 79\penalty0 (4):\penalty0
  527--538, 1997.
\newblock \doi{10.1162/003465397557015}.

\bibitem[Belloni et~al.(2015)Belloni, Chernozhukov, Chetverikov, and
  Kato]{belloni2015some}
A.~Belloni, V.~Chernozhukov, D.~Chetverikov, and K.~Kato.
\newblock Some new asymptotic theory for least squares series: Pointwise and
  uniform results.
\newblock \emph{Journal of Econometrics}, 186\penalty0 (2):\penalty0 345--366,
  2015.
\newblock \doi{10.1016/j.jeconom.2015.02.014}.

\bibitem[Benatia et~al.(2017)Benatia, Carrasco, and
  Florens]{benatia2017functional}
D.~Benatia, M.~Carrasco, and J.-P. Florens.
\newblock Functional linear regression with functional response.
\newblock \emph{Journal of Econometrics}, 201\penalty0 (2):\penalty0 269--291,
  2017.
\newblock \doi{10.1016/j.jeconom.2017.08.008}.

\bibitem[Bissantz et~al.(2007)Bissantz, Dümbgen, Holzmann, and
  Munk]{bissantz2007non}
N.~Bissantz, L.~Dümbgen, H.~Holzmann, and A.~Munk.
\newblock Non-parametric confidence bands in deconvolution density estimation.
\newblock \emph{Journal of the Royal Statistical Society. Series B: Statistical
  Methodology}, 69\penalty0 (3):\penalty0 483--506, 2007.
\newblock \doi{10.1111/j.1467-9868.2007.599.x}.

\bibitem[Blundell et~al.(2007)Blundell, Chen, and Kristensen]{blundell2007semi}
R.~Blundell, X.~Chen, and D.~Kristensen.
\newblock Semi-nonparametric iv estimation of shape-invariant engel curves.
\newblock \emph{Econometrica}, 75\penalty0 (6):\penalty0 1613--1669, 2007.
\newblock \doi{10.1111/j.1468-0262.2007.00808.x}.

\bibitem[Bonhomme and Robin(2010)]{bonhomme2010generalized}
S.~Bonhomme and J.-M. Robin.
\newblock Generalized non-parametric deconvolution with an application to
  earnings dynamics.
\newblock \emph{Review of Economic Studies}, 77\penalty0 (2):\penalty0
  491--533, 2010.
\newblock \doi{10.1111/j.1467-937X.2009.00577.x}.

\bibitem[Boucheron et~al.(2013)Boucheron, Lugosi, and
  Massart]{boucheron2013concentration}
S.~Boucheron, G.~Lugosi, and P.~Massart.
\newblock \emph{Concentration Inequalities}.
\newblock Oxford University Press, 2013.
\newblock \doi{10.1093/acprof:oso/9780199535255.001.0001}.

\bibitem[Cardot et~al.(2003)Cardot, Ferraty, Mas, and Sarda]{cardot2003testing}
H.~Cardot, F.~Ferraty, A.~Mas, and P.~Sarda.
\newblock Testing hypotheses in the functional linear model.
\newblock \emph{Scandinavian Journal of Statistics}, 30\penalty0 (1):\penalty0
  241--255, 2003.
\newblock \doi{10.1111/1467-9469.00329}.

\bibitem[Cardot et~al.(2007)Cardot, Mas, and Sarda]{cardot2007clt}
H.~Cardot, A.~Mas, and P.~Sarda.
\newblock Clt in functional linear regression models.
\newblock \emph{Probability Theory and Related Fields}, 138\penalty0
  (3-4):\penalty0 325--361, 2007.
\newblock \doi{10.1007/s00440-006-0025-2}.

\bibitem[Carrasco and Florens(2011)]{carrasco2011spectral}
M.~Carrasco and J.-P. Florens.
\newblock A spectral method for deconvolving a density.
\newblock \emph{Econometric Theory}, 27\penalty0 (3):\penalty0 546--581, 2011.
\newblock \doi{10.1017/S026646661000040X}.

\bibitem[Carrasco et~al.(2007)Carrasco, Florens, and
  Renault]{carrasco2007linear}
M.~Carrasco, J.-P. Florens, and E.~Renault.
\newblock Chapter 77: Linear inverse problems in structural econometrics
  estimation based on spectral decomposition and regularization.
\newblock \emph{Handbook of Econometrics}, 6:\penalty0 5633--5751, 2007.
\newblock \doi{10.1016/S1573-4412(07)06077-1}.

\bibitem[Carrasco et~al.(2014)Carrasco, Florens, and
  Renault]{carrasco2013asymptotic}
M.~Carrasco, J.-P. Florens, and E.~Renault.
\newblock Asymptotic normal inference in linear inverse problems.
\newblock \emph{The Oxford Handbook of Applied Nonparametric and Semiparametric
  Econometrics and Statistics}, 2014.
\newblock \doi{10.1093/oxfordhb/9780199857944.013.003}.

\bibitem[Centorrino(2016)]{centorrino2014data}
S.~Centorrino.
\newblock Data-driven selection of the regularization parameter in additive
  nonparametric instrumental regressions.
\newblock \emph{Economics Department, Stony Brook University}, 2016.

\bibitem[Chen and Christensen(2018)]{chen2018optimal}
X.~Chen and T.~Christensen.
\newblock Optimal sup-norm rates and uniform inference on nonlinear functionals
  of nonparametric iv regression.
\newblock \emph{Quantitative Economics}, 9\penalty0 (1):\penalty0 39--84, 2018.
\newblock \doi{10.3982/QE722}.

\bibitem[Chen and Pouzo(2015)]{chen2015sieve}
X.~Chen and D.~Pouzo.
\newblock Sieve wald and qlr inferences on semi/nonparametric conditional
  moment models.
\newblock \emph{Econometrica}, 83\penalty0 (3):\penalty0 1013--1079, 2015.
\newblock \doi{10.3982/ECTA10771}.

\bibitem[Chen and Reiss(2011)]{chen2011rate}
X.~Chen and M.~Reiss.
\newblock On rate optimality for ill-posed inverse problems in econometrics.
\newblock \emph{Econometric Theory}, 27\penalty0 (3):\penalty0 497--521, 2011.
\newblock \doi{10.1017/S0266466610000381}.

\bibitem[Chernozhukov et~al.(2014)Chernozhukov, Chetverikov, and
  Kato]{chernozhukov2014gaussian}
V.~Chernozhukov, D.~Chetverikov, and K.~Kato.
\newblock Gaussian approximation of suprema of empirical processes.
\newblock \emph{Annals of Statistics}, 42\penalty0 (4):\penalty0 1564--1597,
  2014.
\newblock \doi{10.1214/14-AOS1230}.

\bibitem[Chernozhukov et~al.(2016)Chernozhukov, Chetverikov, and
  Kato]{chernozhukov2016empirical}
V.~Chernozhukov, D.~Chetverikov, and K.~Kato.
\newblock Empirical and multiplier bootstraps for suprema of empirical
  processes of increasing complexity, and related gaussian couplings.
\newblock \emph{Stochastic Processes and their Applications}, 126\penalty0
  (12):\penalty0 3632--3651, 2016.
\newblock \doi{10.1016/j.spa.2016.04.009}.

\bibitem[Darolles et~al.(2011)Darolles, Fan, Florens, and
  Renault]{darolles2011nonparametric}
S.~Darolles, Y.~Fan, J.~Florens, and E.~Renault.
\newblock Nonparametric instrumental regression.
\newblock \emph{Econometrica}, 79\penalty0 (5):\penalty0 1541--1565, 2011.
\newblock \doi{10.3982/ECTA6539}.

\bibitem[Dudley()]{dudley2014uniform}
R.~M. Dudley.
\newblock \emph{Uniform Central Limit Theorems}.
\newblock Cambridge University Press.
\newblock \doi{10.1017/CBO9780511665622}.

\bibitem[Dudley(2016)]{dudley2016vn}
R.~M. Dudley.
\newblock V.n. sudakov's work on expected suprema of gaussian processes.
\newblock pages 37--43, 2016.
\newblock \doi{10.1007/978-3-319-40519-3_2}.

\bibitem[Evdokimov(2010)]{evdokimov2010identification}
K.~Evdokimov.
\newblock Identification and estimation of a nonparametric panel data model
  with unobserved heterogeneity.
\newblock \emph{Department of Economics, Princeton University}, 2010.

\bibitem[F\`{e}ve and Florens(2010)]{feve2010practice}
F.~F\`{e}ve and J.~Florens.
\newblock The practice of non-parametric estimation by solving inverse
  problems: The example of transformation models.
\newblock \emph{Econometrics Journal}, 13\penalty0 (3):\penalty0 S1--S27, 2010.
\newblock \doi{10.1111/j.1368-423X.2010.00314.x}.

\bibitem[Florens(2003)]{florens2003inverse}
J.-P. Florens.
\newblock Inverse problems and structural econometrics: The example of
  instrumental variables.
\newblock \emph{Advances in Economics and Econometrics: Theory and
  Applications, Eighth World Congress, Volume II}, pages 284--311, 2003.
\newblock \doi{10.1017/CBO9780511610257.010}.

\bibitem[Florens and Van~Bellegem(2015)]{florens2015instrumental}
J.-P. Florens and S.~Van~Bellegem.
\newblock Instrumental variable estimation in functional linear models.
\newblock \emph{Journal of Econometrics}, 186\penalty0 (2):\penalty0 465--476,
  2015.
\newblock \doi{10.1016/j.jeconom.2015.02.020}.

\bibitem[Florens et~al.(2011)Florens, Johannes, and
  Van~Bellegem]{florens2011identification}
J.-P. Florens, J.~Johannes, and S.~Van~Bellegem.
\newblock Identification and estimation by penalization in nonparametric
  instrumental regression.
\newblock \emph{Econometric Theory}, 27\penalty0 (3):\penalty0 472--496, 2011.
\newblock \doi{10.1017/S026646661000037X}.

\bibitem[Florens et~al.(2017)Florens, Horowitz, and
  Keilegom]{florenshorowitzkeielgom2016}
J.-P. Florens, J.~L. Horowitz, and I.~V. Keilegom.
\newblock Bias-corrected confidence intervals in a class of linear inverse
  problems.
\newblock \emph{Annals of Economics and Statistics}, \penalty0 (128):\penalty0
  203, 2017.
\newblock \doi{10.15609/annaeconstat2009.128.0203}.

\bibitem[Gagliardini and
  Scaillet(2012{\natexlab{a}})]{gagliardini2012nonparametric}
P.~Gagliardini and O.~Scaillet.
\newblock Nonparametric instrumental variable estimation of structural quantile
  effects.
\newblock \emph{Econometrica}, 80\penalty0 (4):\penalty0 1533--1562,
  2012{\natexlab{a}}.
\newblock \doi{10.3982/ECTA7937}.

\bibitem[Gagliardini and Scaillet(2012{\natexlab{b}})]{gagliardini2012tikhonov}
P.~Gagliardini and O.~Scaillet.
\newblock Tikhonov regularization for nonparametric instrumental variable
  estimators.
\newblock \emph{Journal of Econometrics}, 167\penalty0 (1):\penalty0 61--75,
  2012{\natexlab{b}}.
\newblock \doi{10.1016/j.jeconom.2011.08.006}.

\bibitem[Gautier and Kitamura(2013)]{gautier2013nonparametric}
E.~Gautier and Y.~Kitamura.
\newblock Nonparametric estimation in random coefficients binary choice models.
\newblock \emph{Econometrica}, 81\penalty0 (2):\penalty0 581--607, 2013.
\newblock \doi{10.3982/ECTA8675}.

\bibitem[Gin\'{e} and Guillou(2002)]{gine2002rates}
E.~Gin\'{e} and A.~Guillou.
\newblock Rates of strong uniform consistency for multivariate kernel density
  estimators.
\newblock \emph{Annales de l'Institut Henri Poincar\'{e} (B) Probability and
  Statistics}, 38\penalty0 (6):\penalty0 907--921, 2002.
\newblock \doi{10.1016/S0246-0203(02)01128-7}.

\bibitem[Gin\'{e} and Nickl()]{gine2015mathematical}
E.~Gin\'{e} and R.~Nickl.
\newblock Mathematical foundations of infinite-dimensional statistical models.
\newblock \doi{10.1017/CBO9781107337862}.

\bibitem[Hall and Horowitz(2005)]{hall2005nonparametric}
P.~Hall and J.~Horowitz.
\newblock Nonparametric methods for inference in the presence of instrumental
  variables.
\newblock \emph{Annals of Statistics}, 33\penalty0 (6):\penalty0 2904--2929,
  2005.
\newblock \doi{10.1214/009053605000000714}.

\bibitem[Hall and Horowitz(2007)]{hall2007methodology}
P.~Hall and J.~Horowitz.
\newblock Methodology and convergence rates for functional linear regression.
\newblock \emph{Annals of Statistics}, 35\penalty0 (1):\penalty0 70--91, 2007.
\newblock \doi{10.1214/009053606000000957}.

\bibitem[Horowitz and Lee(2012)]{horowitz2012uniform}
J.~Horowitz and S.~Lee.
\newblock Uniform confidence bands for functions estimated nonparametrically
  with instrumental variables.
\newblock \emph{Journal of Econometrics}, 168\penalty0 (2):\penalty0 175--188,
  2012.
\newblock \doi{10.1016/j.jeconom.2011.12.001}.

\bibitem[Kato and Sasaki(2018)]{kato2016uniform}
K.~Kato and Y.~Sasaki.
\newblock Uniform confidence bands in deconvolution with unknown error
  distribution.
\newblock \emph{Journal of Econometrics}, 207\penalty0 (1):\penalty0 129--161,
  2018.
\newblock \doi{10.1016/j.jeconom.2018.07.001}.

\bibitem[Koltchinskii(2001)]{koltchinskii2001rademacher}
V.~Koltchinskii.
\newblock Rademacher penalties and structural risk minimization.
\newblock \emph{IEEE Transactions on Information Theory}, 47\penalty0
  (5):\penalty0 1902--1914, 2001.
\newblock \doi{10.1109/18.930926}.

\bibitem[Koltchinskii(2006)]{koltchinskii2006local}
V.~Koltchinskii.
\newblock Local rademacher complexities and oracle inequalities in risk
  minimization.
\newblock \emph{Annals of Statistics}, 34\penalty0 (6):\penalty0 2593--2656,
  2006.
\newblock \doi{10.1214/009053606000001019}.

\bibitem[Koltchinskii(2011)]{koltchinskii2011oracle}
V.~Koltchinskii.
\newblock \emph{Oracle Inequalities in Empirical Risk Minimization and Sparse
  Recovery Problems: Ecole d’Et{\'e} de Probabilit{\'e}s de Saint-Flour
  XXXVIII-2008}.
\newblock Springer, Berlin Heidelberg, 2011.
\newblock \doi{10.1007/978-3-642-22147-7}.

\bibitem[Li et~al.(1989)]{li1989honest}
K.-C. Li et~al.
\newblock Honest confidence regions for nonparametric regression.
\newblock \emph{The Annals of Statistics}, 17\penalty0 (3):\penalty0
  1001--1008, 1989.
\newblock \doi{10.1214/aos/1176347253}.

\bibitem[Lounici and Nickl(2011)]{lounici2011global}
K.~Lounici and R.~Nickl.
\newblock Global uniform risk bounds for wavelet deconvolution estimators.
\newblock \emph{Annals of Statistics}, 39\penalty0 (1):\penalty0 201--231,
  2011.
\newblock \doi{10.1214/10-AOS836}.

\bibitem[Nair(2009)]{nair2009linear}
M.~T. Nair.
\newblock \emph{Linear Operator Equations}.
\newblock World Scientific, may 2009.
\newblock \doi{10.1142/7055}.

\bibitem[Nakamura et~al.(2016)Nakamura, Steinsson, and
  Liu]{nakamura2014chinese}
E.~Nakamura, J.~Steinsson, and M.~Liu.
\newblock Are chinese growth and inflation too smooth? evidence from engel
  curves.
\newblock \emph{American Economic Journal: Macroeconomics}, 8\penalty0
  (3):\penalty0 113--144, 2016.
\newblock \doi{10.1257/mac.20150074}.

\bibitem[Newey and Powell(2003)]{newey2003instrumental}
W.~Newey and J.~Powell.
\newblock Instrumental variable estimation of nonparametric models.
\newblock \emph{Econometrica}, 71\penalty0 (5):\penalty0 1565--1578, 2003.
\newblock \doi{10.1111/1468-0262.00459}.

\bibitem[Ruymgaart(1998)]{ruymgaart1998note}
F.~Ruymgaart.
\newblock A note on weak convergence of density estimators in hilbert spaces.
\newblock \emph{Statistics}, 30\penalty0 (4):\penalty0 331--343, 1998.
\newblock \doi{10.1080/02331889708802618}.

\bibitem[Tao(2014)]{tao2014inference}
J.~Tao.
\newblock Inference for point and partially identified semi-nonparametric
  conditional moment models.
\newblock \emph{Economics Department, SUniversity of Wisconsin-Madison}, 2014.

\bibitem[Tikhonov(1963)]{tikhonov1963solution}
A.~N. Tikhonov.
\newblock On the solution of ill-posed problems and the method of
  regularization (russian).
\newblock \emph{Doklady Akademii Nauk}, 151\penalty0 (3):\penalty0 501--504,
  1963.

\bibitem[Tsybakov(2009)]{tsybakov2009introduction}
A.~B. Tsybakov.
\newblock \emph{Introduction to nonparametric estimation}.
\newblock Springer, New York, 2009.
\newblock \doi{10.1007/b13794}.

\bibitem[van~der Vaart and Wellner(1996)]{van2000weak}
A.~W. van~der Vaart and J.~A. Wellner.
\newblock \emph{Weak convergence and empirical processes: with applications to
  statistics}.
\newblock Springer, New York, 1996.
\newblock \doi{10.1007/978-1-4757-2545-2}.

\bibitem[Yatchew and No(2001)]{yatchew2001household}
A.~Yatchew and J.~No.
\newblock Household gasoline demand in canada.
\newblock \emph{Econometrica}, 69\penalty0 (6):\penalty0 1697--1709, 2001.
\newblock \doi{10.1111/1468-0262.00264}.

\end{thebibliography}

\newpage
\setcounter{page}{1}
\setcounter{section}{0}
\setcounter{equation}{0}
\setcounter{table}{0}
\setcounter{figure}{0}
\renewcommand{\theequation}{A.\arabic{equation}}
\renewcommand\thetable{A.\arabic{table}}
\renewcommand\thefigure{A.\arabic{figure}}
\renewcommand\thesection{A.\arabic{section}}
\renewcommand\thepage{Appendix - \arabic{page}}
\renewcommand\thetheorem{A.\arabic{theorem}}

\begin{center}
	{\LARGE\textbf{APPENDIX}}	
\end{center}
\bigskip

\section{Concentration of the supremum of empirical processes}\label{app:concentration}
Let $(X_i)_{i\in\Nn}$ be a sequence of i.i.d. random variables taking values in some measure space $(S,\Sc)$, and let $\Hc$ be a countable class of real functions defined on $S$. Consider the empirical process
\begin{equation*}
	\nu_n(h) \equiv \frac{1}{n}\sum_{i=1}^nh(X_i) - \E h(X_i),\qquad h\in\Hc
\end{equation*}
and denote the symmetrized process by
\begin{equation*}
	\nu_n^\eta(h) = \frac{1}{n}\sum_{i=1}^n\eta_ih(X_i),
\end{equation*}
where $(\eta_i)_{i\in\Nn}$ is a sequence of i.i.d. Rademacher random variables, independent from $(X_i)_{i\in\Nn}$. We use $\|.\|_\Hc$ to denote the supremum norm over the class of functions $\Hc$.

\begin{proposition}\label{thm:symmetrization_inequality}
	The following two inequalities hold
	\begin{equation}\label{eq:symmetrization}
	2^{-1}\E \left\|\nu_n^\eta\right\|_\Hc - 2^{-1}n^{-1/2}\|Ph\|_\Hc \leq \E\|\nu_n\|_\Hc \leq 2\E \left\|\nu_n^\eta\right\|_\Hc.
	\end{equation}
\end{proposition}
The second inequality in Eq.~\ref{eq:symmetrization}, is a symmetrization inequality, while the first is the desymmetrization inequality. These inequalities allow us to establish uniform limit theorems and play an essential role in the empirical process theory, see \cite{van2000weak}. We refer to \cite{koltchinskii2006local}, p.7, and references therein for a proof of this modification of the symmetrization inequality.

Our construction of confidence sets relies on the concentration inequality for functions of bounded difference, also known as McDiarmid's inequality, see the excellent treatment of concentration inequalities in \cite{boucheron2013concentration}. Such concentration inequalities tell us that the probability that the supremum of the empirical process $\|\nu_n\|_\Hc$ deviates from its expected value $\E\|\nu_n\|_\Hc$ declines exponentially fast. While this expected value is unknown in practice, it can be estimated using the multiplier bootstrap with Rademacher multipliers. The symmetrization inequality allows us then to compare the expected value of this estimate $\|\nu_n^\eta\|_\Hc$ to the original expected value, which leads to the data-driven concentration inequality. This insightful idea comes from the statistical learning literature (see, e.g., \cite{koltchinskii2001rademacher}).

The next proposition states the precise version of the data-driven concentration inequality used in the present paper.
\begin{proposition}\label{thm:supremum_hoeffding}
	Let $(X_i)_{i=1}^n$, $(\eta_i)_{i=1}^n$, and $\Hc$ be defined as before. Suppose also that the absolute value of all $h\in\Hc$ is uniformly bounded by some $H$. Then for all $x>0$ and $n\in\Nn$
	\begin{equation*}
		\Pr\left(\|\nu_n\|_\Hc > 2\left\|\nu_n^\eta\right\|_\Hc + 3H\sqrt{\frac{2x}{n}}\right) \leq e^{-x}.
	\end{equation*}
\end{proposition}
See \cite{koltchinskii2011oracle}, Theorem 4.6, for the proof of this result.

\section{General ill-posed inverse problems}
\subsection{Impossibility of weak convergence}
We first discuss the impossibility of using the uniform central limit theorem to obtain the distribution of the estimator. 

To fix notation through the rest of the paper, let $P$ be a probability measure to any of the ill-posed models introduced in the previous section. The model is described by the function equation $r=T\varphi$, where $\varphi:[0,1]^p\to\R$ is an infinite dimensional parameter of interest. We aim to construct a random set $C_{n,1-\gamma}$ containing the function $\varphi$ with probability of at least $1-\gamma$ for $\gamma\in(0,1)$ and such that its diameter shrinks as the sample size increases at a specific rate. We focus on confidence sets for the Tikhonov-regularized estimator defined as follows
\begin{equation*}
	\hat{\varphi} = (\alpha_n I + \hat T^*\hat T)^{-1}\hat T^*\hat r,
\end{equation*}
where $\hat T,\hat T^*$, and $\hat r$ are appropriate estimators\footnote{If some operators are known, which is the case in the density deconvolution model, we replace estimators by known quantities.} and $\alpha_n$ is some positive sequence converging to zero as $n\to\infty$. This estimator belongs to the general family of spectral regularization schemes, see \cite{carrasco2007linear}.

For models considered in this paper, the dominating stochastic component of the Tikhonov-regularized estimator is driven by a sequence of i.i.d. centered random functions
\begin{equation}\label{eq:processes}
\nu_{n} = \frac{1}{n}\sum_{i=1}^n(\alpha_nI + T^*T)^{-1}T^*X_{ni}.
\end{equation}
Assume that $T$ is an integral operator with continuous kernel function, so that $T,T^*$, and $T^*T$ map to the space of continuous functions. Thus, we can also think of the operator $T$ as acting between $(C,\|.\|_\infty)$ spaces. By Lemma~\ref{lemma:infinity_to_hilbert} in the Appendix C, the operator $(\alpha_n I + T^*T)$ is invertible between spaces of continuous functions. Thus, trajectories of the process in Eq.~\ref{eq:processes} belong to the space $(C,\|.\|_\infty)$.

To build uniform confidence sets, we need to approximate the distribution of the supremum of this process. The most straightforward route to achieve this would be to establish its weak convergence to some Gaussian process and then to rely on quantiles of Gaussian suprema to build uniform confidence sets. Unfortunately, the process in Eq.~\ref{eq:processes} can not converge weakly as a random element in $(C,\|.\|_\infty)$ space. Indeed, we introduced regularization to smooth out discontinuities of the operator inversion. However, in the limit, as the regularization parameter tends to zero, we get back a discontinuous/unbounded operator. This discontinuity in the limit is a precise reason why the functional convergence fails.

More precisely, we can see that depending on the direction $\delta\in L_2$, the speed of convergence of inner products $\langle\hat\varphi - \varphi,\delta\rangle$ differs, see \cite{carrasco2007linear}, which makes it impossible to converge weakly in $L_2$. In Proposition~\ref{prop:impossibility} we show formally that the weak convergence in $(C,\|.\|_\infty)$ is impossible generalizing the result of \cite{cardot2007clt} who focus on the weak convergence in $L_2$ in the special case of the functional linear regression without instrumental variables. Note that our result holds for the arbitrary ill-posed inverse model, including the nonparametric IV, functional IV regressions, and the density deconvolution. Note also that our result holds regardless of what type of the estimator is used in practice, be it kernel or series estimators in the nonparametric IV model. In particular, it is worth comparing it to the impossibility of weak convergence of the kernel density estimator, where inner products do converge at the same speed, but the only compatible limiting process is a white noise process, as in \cite{ruymgaart1998note}.

\begin{proposition}\label{prop:impossibility}
	Suppose that the inverse of the operator $T^*T$ in Eq.~\ref{eq:processes} is unbounded. Then there does not exist a normalizing sequence $r_n$ such that $r_n\nu_{n}$ would converge weakly in $(C,\|.\|_\infty)$ to a non-degenerate random process.
\end{proposition}
\begin{proof}[Proof of Proposition~\ref{prop:impossibility}]
	Recall that for some normalizing sequence $r_n$, the weak convergence of $r_n\nu_{n}$ in $L_2$ to some random element requires $\langle r_n\nu_{n},\delta\rangle$ to converge weakly in $\R$ for all $\delta\in L_2$, \cite{van2000weak}, Theorem 1.8.4. Since $(T^*T)^{-1}$ is unbounded with $\Dc[(T^*T)^{-1}]\subset L_2[0,1]^p$, for $\delta\in \Dc[(T^*T)^{-1}]$, we need to set $r_n=n^{1/2}$, since
	\begin{equation*}
		\left\langle r_n\nu_{n},\delta\right\rangle = \frac{1}{\sqrt{n}}\sum_{i=1}^n\left\langle T^*X_{ni},(\alpha_n I + T^*T)^{-1}\delta\right\rangle\xrightarrow{d} N\left(0,\E\left\langle T^*X_1,(T^*T)^{-1}\delta\right\rangle^2\right).
	\end{equation*}
	On the other hand, if $\delta\not\in \Dc[(T^*T)^{-1}]$, $\|(\alpha_nI + T^*T)^{-1}\delta\|^2\to\infty$, making $\E\langle T^*X_1,(\alpha_nI + T^*T)^{-1}\delta\rangle^2\to\infty$, and so $\langle n^{1/2}\nu_{n},\delta\rangle$ can not converge in distribution. This shows that it is not possible to converge weakly in $L_2$. Since bounded and continuous functionals on $L_2$ are bounded and continuous on $(C,\|.\|_\infty)$ for finite measure spaces, it follows from the definition of weak convergence that $r_n\nu_{n}$ does not converge weakly in $(C,\|.\|_\infty)$ for any choice of the normalizing sequence $r_n$.
\end{proof}
Proposition~\ref{prop:impossibility} tells us that it is not possible to rely on the asymptotic approximation with the conventional central limit theorem. As a result, we need to rely on alternative approaches to inference.

\subsection{Bootstrap confidence sets}
Consider the empirical process
\begin{equation*}
	V_{n}(g) = \sqrt{n}(P_n-P)g,\qquad g\in\Gc_n = \left\{g_w:\; w\in[0,1]^q\right\}
\end{equation*}
indexed by some pointwise measurable pre-Gaussian class of functions $\Gc_n\subset L_2(P)$, where $L_2(P)$ is the set of square-integrable functions with respect to $P$ and $P_n$ is the empirical measure. We will always assume that the data are i.i.d. In our setting the empirical process $V_{n}$ can also be considered as a stochastic process
\begin{equation*}
	V_{n}(g_w) = \frac{1}{\sqrt{n}}\sum_{i=1}^nX_{ni}(w)
\end{equation*}
indexed by $w\in[0,1]^q$. Let
\begin{equation*}
	\|V_{n}\|_\infty \triangleq \|V_{n}\|_{\Gc_n} \triangleq \sup_{g\in\Gc_n}|V_{n}(g)|
\end{equation*}
be the supremum of the empirical process and let
\begin{equation*}
	\|\Gbb_n\|_\infty \triangleq \|\Gbb_n\|_{\Gc_n} \triangleq \sup_{g\in\Gc_n}|\Gbb_n(g)|
\end{equation*}
be the supremum of the tight centered Gaussian process with the same covariance structure as $V_{n}$, i.e., for all $g_1,g_2\in\Gc_n$
\begin{equation*}
	\begin{aligned}
		\E[\Gbb_n(g_1)\Gbb_n(g_2)] & = \E[V_{n}(g_1)V_{n}(g_2)] \\
		& = P(g_1-Pg_1)(g_2-Pg_2).
	\end{aligned}
\end{equation*}
We shall note that we will work only with processes having continuous trajectories so that all suprema can be restricted to countable sets of rational numbers, and whence are measurable.

Note that by the linearity
\begin{equation*}
	\sup_{g\in\Gc_n}\left|V_{n}(g)\right| = \sup_{g\in\Gc_n\cup-\Gc_n}V_{n}(g),
\end{equation*}
so we can always drop the absolute value from the supremum extending the supremum to the enlarged class of functions. A similar observation is valid for the supremum of the Gaussian process \cite{dudley2014uniform}, Theorem 2.1. We will typically work with suprema of absolute value and will use the above facts repeatedly.

Let $G_{n}$ be the envelope of the class $\Gc_n$, which we assume to be square-integrable and let $N(\Gc_n,\|.\|_{Q,2},\epsilon)$ be the $\epsilon$-covering of $(\Gc_n,\|.\|_{Q,2})$, with $\|g\|_{Q,q} = (Q|g|^q)^{1/q}$ for some probability measure $Q$. We assume that $\Gc_n$ is a VC-type class of functions.
\begin{assumption}\label{as:vc_type}
	Suppose that the class $\Gc_n$ is such that
	\begin{equation*}
		\sup_{Q}N(\Gc_n,\|.\|_{Q,2},\epsilon\|G_{n}\|_{Q,2})\leq \left(\frac{A}{\epsilon}\right)^v,\qquad 0<\epsilon\leq 1
	\end{equation*}
	for some universal constants $A\geq e$ and $v\geq 1$, where the supremum is taken over all probability measures $Q$ with $0<\|G_{n}\|_{Q,2} <\infty$.
\end{assumption}

For an i.i.d. sequence of Gaussian random variables $(\varepsilon_i)_{i=1}^n$, consider the supremum of the bootstrapped process
\begin{equation*}
	\left\|V_{n}^\varepsilon\right\|_\infty = \left\|\frac{1}{\sqrt{n}}\sum_{i=1}^n\varepsilon_iX_{ni}\right\|_\infty.
\end{equation*}

The process $X_{ni}$ is typically unobserved. Let $\hat X_{ni}$ be its empirical counterpart estimated from the data and let $c_{1-\gamma}^*$ be $1-\gamma$ conditional on the data $\Dc_n$ quantile of the supremum of the bootstrapped process
\begin{equation*}
	\|\hat V_{n}^\varepsilon\|_\infty = \left\|\frac{1}{\sqrt{n}}\sum_{i=1}^n\varepsilon_i\hat X_{ni}\right\|_\infty,
\end{equation*}
where $(\varepsilon_i)_{i=1}^n$ are i.i.d. $N(0,1)$ random variables independent from the sample.

\begin{assumption}\label{as:variance_bounds}
	(i) there exist constants $0<\underline\sigma<\bar{\sigma}$ such that $\underline{\sigma}^2\leq \Var(X_{ni}(w)) \leq \bar{\sigma}^2$ for all $w\in[0,1]^q$; (ii) $\sup_{g\in\Gc_n}\|g\|_{P,k}^r\leq \sigma^2b_n^{r-2}<\infty,r=2,3,4$ and $\|G_{n}\|_{P,4}\leq b_n < \infty$ for some $b_n=O(n^c)$.
\end{assumption}
Assumptions~\ref{as:probability_rates} and \ref{as:variance_bounds} consist of several technical conditions needed for the Gaussian and bootstrap approximation of \cite{chernozhukov2014gaussian}.

\begin{assumption}\label{as:probability_rates}
	Let $\Pc$ be the class of models consisting of all probability distributions corresponding to the ill-posed model satisfying Assumptions~\ref{as:source_condition1}, \ref{as:vc_type}, \ref{as:variance_bounds}. Suppose additionally that $\Pc$ is such that
	\begin{equation*}
		\begin{aligned}
			(C1) &\quad \sup_{P\in\Pc}\E_P\left\|\hat T^* - T^*\right\|_{2,\infty}^2 = O(\delta_{1n}^2), \\
			(C2) &\quad \sup_{P\in\Pc}\E_P\left\|\hat V_n^\varepsilon - V^\varepsilon_n\right\|_\infty = O(\delta_{2n}), \\	
			(C3) &\quad \hat r - \hat T\varphi = \frac{u_n}{n}\sum_{i=1}^nX_{ni} + R_{1n},
		\end{aligned}
	\end{equation*}
	where $R_{1n}$ is such that $\sup_{P\in\Pc}\E_P\|R_{1n}\|^2 = O(\delta_{3n}^2)$ for some sequences $\delta_{1n},\delta_{2n},\delta_{3n}\to 0$.
\end{assumption}
Assumption (C3) is asymptotic linearization of residuals in the model. In some cases this expansion is exact and $R_{1n}$ is trivially zero. (C1) controls the estimation error in the operator, while (C2) controls the error in the bootstrap process.

We focus on bootstrap confidence sets described as
\begin{equation}\label{eq:boostrap_cs}
C_{n,1-\gamma}^* = \left\{\left[\hat \varphi(z) - q_n^*,\; \hat\varphi(z) + q_n^*\right]:\; z\in[0,1]^p\right\}
\end{equation}
with $q_n^* = \frac{c_{1-\gamma}^*\|\hat T^*\|_{2,\infty} + \log^{-1}n}{\alpha_nn^{1/2}}u_n$.

In the following theorem, we show that our bootstrap confidence sets are honest to $\Pc$ and characterize coverage errors explicitly. This result will be specialized to each particular model in the following sections.
\begin{theorem}\label{thm:bootstrap}
	Suppose that Assumptions~\ref{as:source_condition1}, \ref{as:vc_type}, \ref{as:variance_bounds}, and \ref{as:probability_rates} are satisfied and $b_n^4\log^7n/n\to 0$. Then
	\begin{equation*}
		\inf_{P\in\Pc}\Pr(\varphi\in C_{n,1-\gamma}^*) \geq 1 - \gamma - O_p\left(\delta_n\right) - o_p(1)
	\end{equation*}
	with
	\begin{equation*}
		\delta_n = \left(\alpha_n^{-1/2}\delta_{1n} + \alpha_n^{-1}\delta_{3n} + \alpha_n^{\beta\wedge 1}\right)u_n^{-1}\alpha_nn^{1/2}\log n + \delta_{2n}\log n.
	\end{equation*}
\end{theorem}
\begin{proof}
	Using Assumption~\ref{as:probability_rates} (C3), decompose
	\begin{equation*}
		\hat\varphi - \varphi = \hat\nu_n + (\alpha_nI + \hat T^*\hat T)^{-1}\hat T^*R_{1n} + R_{2n} + R_{3n}
	\end{equation*}
	with
	\begin{equation*}
		\begin{aligned}
			\hat\nu_n & = (\alpha_n I + \hat T^*\hat T)^{-1}\hat T^*\frac{u_n}{n}\sum_{i=1}^nX_{ni} \\
			R_{2n} & = \left[(\alpha_n I + \hat T^*\hat T)^{-1}\hat T^*\hat T - (\alpha_n I + T^*T)^{-1}T^*T\right]\varphi.  \\
			R_{3n} & = (\alpha_n I + T^*T)^{-1}T^*T\varphi - \varphi.
		\end{aligned}
	\end{equation*}
	By Lemma~\ref{lemma:xi_term} there exists a constant $C<\infty$ independent from $P\in\Pc$ such that
	\begin{equation*}
		\|R_{2n}\|_\infty \leq \frac{C}{\alpha_n^{1/2}}\left(\|\hat T^* - T^*\|_{2,\infty} + \|\hat T^* - T^*\|_{2,\infty}^2\right),
	\end{equation*}
	whence under Assumption~\ref{as:probability_rates} (C1)
	\begin{equation*}
		\sup_{P\in\Pc}\E_P\|R_{2n}\|_\infty = O\left(\frac{\delta_{1n} + \delta_{1n}^2}{\alpha_n^{1/2}}\right).
	\end{equation*}
	Next, under Assumption~\ref{as:source_condition1} by Proposition~\ref{prop:regularization_bias} there exists a constant $R$ independent from $P\in\Pc$ such that
	\begin{equation*}
		\|R_{3n}\|_\infty \leq R\alpha_n^{\beta\wedge 1}.
	\end{equation*}
	Lastly,
	\begin{equation*}
		\begin{aligned}
			\left\|(\alpha_n I + \hat T^*\hat T)^{-1}\hat T^*R_{1n}\right\|_\infty & = \left\|\hat T^*(\alpha_n I + \hat T\hat T^*)^{-1}R_{1n}\right\|_\infty \\
			& \leq \frac{\|\hat T^*\|_{2,\infty}}{\alpha_n}\|R_{1n}\|.
		\end{aligned}
	\end{equation*}
	
	Then under Assumption~\ref{as:probability_rates} by Markov's inequality
	\begin{equation*}
		\begin{aligned}
			& \Pr\left(\varphi\in C_{n,1-\gamma}^*\right) \\
			& = \Pr\left(\|\hat \varphi - \varphi\|_\infty\leq q_n^*\right) \\
			& \geq \Pr\left(\|\hat\nu_n\|_\infty + \frac{\|\hat T^*\|_{2,\infty}}{\alpha_n}\|R_{1n}\| + \|R_{2n}\|_\infty + \|R_{3n}\|_\infty\leq q_n^* \right) \\
			& \geq \Pr\left(\left\|V_{n}\right\|_\infty \leq c_{1-\gamma}^* + (2\|\hat T^*\|_{2,\infty}\log n)^{-1}\right) \\
			& \qquad - \Pr\left(\frac{\|\hat T^*\|_{2,\infty}}{\alpha_n}\|R_{1n}\| + \|R_{2n}\|_\infty + \|R_{3n}\|_\infty > \frac{u_n}{2\alpha_nn^{1/2}\log n}\right) \\
			& \geq \Pr\left(\left\|V_{n}\right\|_\infty \leq c_{1-\gamma}^* + (2\|\hat T^*\|_{2,\infty}\log n)^{-1}\right)\\
			& \qquad - \sup_{P\in\Pc}\E_P\left[\frac{\|\hat T^*\|_{2,\infty}}{\alpha_n}\|R_{1n}\| + \|R_{2n}\|_\infty + \|R_{3n}\|_\infty\right]u_n^{-1}\alpha_nn^{1/2}\log n \\
			& \geq \Pr\left(\left\|V_{n}\right\|_\infty \leq c_{1-\gamma}^* + (2\|\hat T^*\|_{2,\infty}\log n)^{-1}\right) \\
			& \qquad - O\left(\left(\alpha_n^{-1}\delta_{3n} + \alpha_n^{-1/2}\delta_{1n} + \alpha_n^{\beta\wedge 1}\right)u_n^{-1}\alpha_nn^{1/2}\log n\right). \\
		\end{aligned}
	\end{equation*}
	
	Under Assumptions~\ref{as:vc_type} and \ref{as:variance_bounds} by the coupling inequality in \cite{chernozhukov2016empirical}, Theorem 2.1, for any $P\in\Pc$ one can approximate the supremum of the empirical process $V_{n}$ by the supremum of the Gaussian process $\Gbb_n$ with the same covariance structure
	\begin{equation*}
		\begin{aligned}
			& \Pr\left(\|V_{n}\|_\infty\leq c^*_{1-\gamma} + (2\|\hat T^*\|_{2,\infty}\log n)^{-1}\right) \\
			& \geq \Pr\left(\left\|\Gbb_n\right\|_{\Gc_n} \leq c_{1-\gamma}^* + (2\|\hat T^*\|_{2,\infty}\log n)^{-1} -  C_1\left(\frac{b_n\log^{5/4} n}{n^{1/4}} + \frac{b_n^{1/3}\log n}{n^{1/6}}\right)\right) \\
			& \qquad - \Pr\left(\left|\|V_n\|_\infty - \|\Gbb_n\|_{\Gc_n}\right| > C_1\left(\frac{b_n\log^{5/4} n}{n^{1/4}} + \frac{b_n^{1/3}\log n}{n^{1/6}}\right) \right) \\
			& \geq \Pr\left(\left\|\Gbb_n\right\|_{\Gc_n} \leq c_{1-\gamma}^* + (2\|\hat T^*\|_{2,\infty}\log n)^{-1} -  C_1\left(\frac{b_n\log^{5/4} n}{n^{1/4}} + \frac{b_n^{1/3}\log n}{n^{1/6}}\right)\right) - o(1),
		\end{aligned}
	\end{equation*}
	where $C_1$ is some absolute constants and the $o(1)$ term does not depend on $P\in\Pc$. Next, for some sequence of polynomial order $\delta_{4n}$ (to be specified below), put
	\begin{equation*}
		\begin{aligned}
			\delta_n^I & = \delta_{4n} + \frac{b_n\log^{5/4} n}{n^{1/4}} + \frac{b_n^{1/3}\log n}{n^{1/6}} \\
			\delta_n^{II} & = \left(\alpha_n^{-1}\delta_{3n} + \alpha_n^{-1/2}\delta_{1n} + \alpha_n^{\beta\wedge 1}\right)u_n^{-1}\alpha_nn^{1/2}\log n.
		\end{aligned}
	\end{equation*}
	Then
	\begin{equation*}
		\begin{aligned}
			\Pr\left(\varphi\in C^*_{n,1-\gamma}\right)
			& \geq \Pr\left(\|\Gbb_n\|_{\Gc_n} \leq c^*_{1-\gamma} + (2\|\hat T^*\|_{2,\infty}\log n)^{-1} + \delta_{4n}\right) \\
			& - \sup_{x\in\R}\Pr\left(|\|\Gbb_n\|_{\Gc_n} - x|\leq \delta_n^{I}\right) - O\left(\delta^{II}_n\right) - o(1).
		\end{aligned}
	\end{equation*}
	Under Assumption~\ref{as:variance_bounds} (i) by the anti-concentration inequality in \cite{chernozhukov2014gaussian}, Lemma A.1, for every $\epsilon>0$
	\begin{equation*}
		\sup_{x\in\R}\Pr\left(\left|\|\Gbb_n\|_{\Gc_n} - x\right|\leq \epsilon\right)\leq C_2\epsilon\left(\E\|\Gbb_n\|_{\Gc_n} + \sqrt{1\vee \log(\underline{\sigma}/\epsilon)}\right),
	\end{equation*}
	where the constant $C_2$ depends only on $\underline{\sigma}$ and $\bar\sigma$.
	
	Moreover, under Assumption~\ref{as:variance_bounds} by the Dudley-Sudakov's entropy bound\footnote{Note that $N(\Gc_n,\|.\|_{P,2},\epsilon)=1$ for all $\epsilon\geq \mathrm{diam}(\Gc_n)/2$ and $\mathrm{diam}(\Gc_n)\leq 2\sup_{g\in\Gc_{n}}\|g\|_{P,2}\leq 2\bar\sigma$.}, see e.g. \cite{dudley2016vn}
	\begin{equation*}
		\begin{aligned}
			\E\|\Gbb_n\|_{\Gc_n} & \leq 24\int_0^{\bar{\sigma}}\sqrt{\log N(\Gc_n,\|.\|_{P,2},\epsilon)}\dx \epsilon =O\left(\log^{1/2} n\right),
		\end{aligned}
	\end{equation*}
	where the last equality follows under Assumptions~\ref{as:vc_type} and \ref{as:variance_bounds} (ii).
	
	Next under Assumptions~\ref{as:vc_type} and \ref{as:variance_bounds} by the coupling inequality in \cite{chernozhukov2016empirical}, Theorem 2.2, we approximate the supremum of the Gaussian process by the supremum of the multiplier bootstrap process. To that end for some constant $C_3$ independent from $P\in\Pc$
	\begin{equation*}
		\begin{aligned}			
			&\Pr\left(\|\Gbb_n\|_\infty \leq c^*_{1-\gamma} + (2\|\hat T^*\|_{2,\infty}\log n)^{-1} + C_3\left(\frac{b_n\log^{5/4} n}{n^{1/4}} + \frac{b_n^{1/2}\log n}{n^{1/4}} \right)\right) \\
			& \geq \Pr\left(\|V_{n}^\varepsilon\|_\infty \leq c^*_{1-\gamma} + (2\|\hat T^*\|_{2,\infty}\log n)^{-1}| \Dc_n\right) - o_p(1) \\
			& \geq \Pr\left(\|\hat V_{n}^\varepsilon\|_\infty \leq c^*_{1-\gamma}|\Dc_n\right) - 2\|\hat T^*\|_{2,\infty}\E\left[\left\|\hat V_n^\varepsilon - V_n^\varepsilon\right\|_\infty|\Dc_n\right]\log n - o_p(1) \\
		\end{aligned}
	\end{equation*}
	where we use Markov's inequality, $\Dc_n$ denotes the data, and the $o_p(1)$ term does not depend on $P\in\Pc$. Setting $\delta_{4n} = C_3\left(\frac{b_n\log^{5/4} n}{n^{1/4}} + \frac{b_n^{1/2}\log n}{n^{1/4}}\right)$ under Assumption~\ref{as:probability_rates} (C2) we obtain
	\begin{equation*}
		\Pr\left(\varphi\in C^*_{n,1-\gamma}\right) \geq 1-\gamma - O_p\left(\delta_n^I\log^{1/2} n + \delta_n^{II} + \delta_{2n}\log n\right) - o_p(1),
	\end{equation*}
	whence the result follows under the assumption $b_n^4\log^7n/n\to 0$.
\end{proof}

To characterize expected diameters of bootstrap confidence sets we need additional assumptions.
\begin{assumption}\label{as:bootstrap_cs_diameters}
	Suppose that (i) conditionally on the data $\|\hat V_n^\varepsilon\|_\infty$ is a supremum over the class of functions satisfying Assumption~\ref{as:vc_type};
	\begin{equation*}
		(ii)\;\sup_{P\in\Pc}\max_{1\leq i\leq n}\E_P\|\hat X_{ni} - X_{ni}\|^2_\infty = O(1);\qquad (iii)\; \sup_{P\in\Pc}\max_{1\leq i\leq n}\E_P\|X_{ni}\|^2_\infty = O(u_n^2),
	\end{equation*}
	and $u_n\to \infty$.
\end{assumption}
\begin{theorem}\label{thm:exp_diam_bootstrap}
	Suppose that Assumption~\ref{as:bootstrap_cs_diameters} and assumptions of Theorem~\ref{thm:bootstrap} are satisfied. If $\delta_{1n}=O(1)$, then
	\begin{equation*}
		\left|C^*_{n,1-\gamma}\right| = O_p\left(\frac{u_n^2}{\alpha_nn^{1/2}}\right)
	\end{equation*}
	uniformly over $P\in\Pc$.
\end{theorem}
\begin{proof}[Proof of Theorem~\ref{thm:exp_diam_bootstrap}]
	First note that
	\begin{equation*}
		\begin{aligned}
			\left|C_{n,1-\gamma}^*\right|_\infty & = 2\left|q_n^*\right| \\
			& = \frac{2u_n}{\alpha_nn^{1/2}}\left\{\left[c^*_{1-\gamma}\|\hat T^*\|_{2,\infty}\right] + \log^{-1}n \right\}.
		\end{aligned}
	\end{equation*}
	Recall that $c^*_{1-\gamma}$ is $1-\gamma$ conditional on the data $\Dc_n$ quantile of the supremum of the Gaussian process $\|\hat V_n^\varepsilon\|_\infty$. By the Gaussian concentration inequality, e.g., see \cite{boucheron2013concentration}, Theorem 5.8, conditionally on the data
	\begin{equation*}
		c_{1-\gamma}^* \leq \E_P\left[\|\hat V_n^\varepsilon\|_\infty | \Dc_n\right] + \sqrt{\frac{2}{n}\sum_{i=1}^n\|\hat X_{ni}\|_\infty^2\log(1/\gamma)}.
	\end{equation*}
	By Dudley-Sudakov's inequality, e.g., see \cite{dudley2016vn}, and the change of variables under Assumption~\ref{as:vc_type}
	\begin{equation*}
		\E_P\left[\|\hat V_n^\varepsilon\|_\infty | \Dc_n\right] \leq 24\sqrt{\frac{1}{n}\sum_{i=1}^n\|\hat X_{ni}\|_\infty^2}\int_0^1\sqrt{v\log\frac{A}{\varepsilon}}\dx\varepsilon.
	\end{equation*}
	The result follows under Assumption~\ref{as:probability_rates} (C1).
\end{proof}

\subsection{Concentration-based confidence sets}
The confidence sets based on the concentration inequality are described as
\begin{equation*}
	C_{n,1-\gamma} = \left\{\left[\hat \varphi(z) - \hat q_n,\; \hat\varphi(z) + \hat q_n\right]:\; z\in[0,1]^p\right\},
\end{equation*}
where $\hat q_n = 2\left\|\hat \nu_{n}^\eta\right\|_\infty +  \frac{3\|\hat T^*\|_{2,\infty}G\sqrt{2\log(1/\gamma)} + \log^{-1} n}{\alpha_nn^{1/2}}u_n$ and
\begin{equation*}
	\left\|\hat \nu_{n}^\eta\right\|_\infty = \left\|\frac{u_n}{n}\sum_{i=1}^n(\alpha_n I + \hat T^*\hat T)^{-1}\hat T^*\eta_i\hat X_{ni}\right\|_\infty
\end{equation*}
denotes the estimate of the supremum of the symmetrized process. The operator norm $\|\hat T^*\|_{2,\infty}$ can be computed by Lemma~\ref{lemma:operator_to_kernel}, while $G$ is selected so that $\|X_{ni}\|\leq G$.

For concentration inequality we need a somewhat different version of Assumption~\ref{as:probability_rates}.
\begin{assumption}\label{as:probability_rates_ci}
	Let $\Pc$ be the class of models consisting of all probability distributions corresponding to the ill-posed model satisfying Assumption~\ref{as:source_condition1}. Suppose additionally that $\Pc$ is such that
	\begin{equation*}
		\begin{aligned}
			(C1) &\quad \sup_{P\in\Pc}\E_P\left\|\hat T^* - T^*\right\|_{2,\infty}^2 = O(\delta_{1n}^2), \\
			(C2) &\quad \sup_{P\in\Pc}\max_{1\leq i\leq n}\E_P\left\|\hat X_{ni} - X_{ni}\right\|^2 = O(\delta_{2n}^2), \\	
			(C3) &\quad \hat r - \hat T\varphi = \frac{u_n}{n}\sum_{i=1}^nX_{ni} + R_{1n},
		\end{aligned}
	\end{equation*}
	where $R_{1n}$ is a remainder term such that $\sup_{P\in\Pc}\E_P\|R_{1n}\|^2 = O(\delta_{3n}^2)$ and $\delta_{1n},\delta_{2n},\delta_{3n}\to 0$ are some sequences.
\end{assumption}

\begin{theorem}\label{thm:main_ci}
	Suppose that Assumption~\ref{as:probability_rates_ci} is satisfied and $\|X_{ni}\|\leq G$ for some universal constant $G$ for all $P\in\Pc$. Then
	\begin{equation*}
		\inf_{P\in\Pc}\Pr\left(\varphi\in C_{n,1-\gamma}\right) \geq 1 - \gamma - O(\delta_n)
	\end{equation*}
	with $\delta_n = \left(\frac{\delta_{3n}}{\alpha_n} + \frac{\delta_{1n}}{\alpha_n^{1/2}} + \alpha_n^{\beta\wedge 1}\right)u_n^{-1}\alpha_nn^{1/2}\log n + \delta_{2n}\log n$.
\end{theorem}
\begin{proof}[Proof of Theorem~\ref{thm:main_ci}]
	Under Assumption~\ref{as:probability_rates_ci} (C3), decompose
	\begin{equation*}
		\hat\varphi - \varphi = \hat\nu_n + (\alpha_nI + \hat T^*\hat T)^{-1}\hat T^*R_{1n} + R_{2n} + R_{3n}
	\end{equation*}
	with
	\begin{equation*}
		\begin{aligned}
			\hat\nu_n & = (\alpha_n I + \hat T^*\hat T)^{-1}\hat T^*\frac{u_n}{n}\sum_{i=1}^nX_{ni} \\
			R_{2n} & = \left[(\alpha_n I + \hat T^*\hat T)^{-1}\hat T^*\hat T - (\alpha_n I + T^*T)^{-1}T^*T\right]\varphi.  \\
			R_{3n} & = (\alpha_n I + T^*T)^{-1}T^*T\varphi - \varphi.
		\end{aligned}
	\end{equation*}
	Recall that $\hat q_n = 2\|\nu_n^\eta\|_\infty + \frac{3\|\hat T^*\|_{2,\infty}G\sqrt{2\log(1/\gamma)} + \log^{-1} n}{\alpha_n n^{1/2}}u_n$ and we also denote
	\begin{equation*}
		\begin{aligned}
			\nu_n & = (\alpha_n I + T^*T)^{-1}T^*\frac{u_n}{n}\sum_{i=1}^nX_{ni}, \\
			\nu_n^\eta & = (\alpha_n I + T^*T)^{-1}T^*\frac{u_n}{n}\sum_{i=1}^n\eta_iX_{ni},
		\end{aligned}
	\end{equation*}
	where $\varepsilon_i$ are i.i.d. Rademacher random variables. Then by the triangle and Markov's inequalities
	\begin{equation}\label{eq:thm_main1.1}
	\begin{aligned}
	& \Pr\left(\varphi\in C_{n,1-\gamma}\right) \\
	& =  \Pr\left(\|\hat \varphi - \varphi\|_\infty\leq \hat q_{n}\right) \\
	& \geq \Pr\left(\|\hat \nu_{n}\|_\infty + \frac{\|\hat T^*\|_{2,\infty}}{\alpha_n}\|R_{1n}\| + \|R_{2n}\|_\infty + \|R_{3n}\|_\infty \leq \hat q_n\right) \\
	& \geq \Pr\left(\|\nu_{n}\|_\infty + \left\|\hat \nu_{n} - \nu_{n}\right\|_\infty + \frac{\|\hat T^*\|_{2,\infty}}{\alpha_n}\|R_{1n}\| + \|R_{2n}\|_\infty + \|R_{3n}\|_\infty\leq \hat q_n\right) \\
	& \geq \Pr\left(\|\nu_{n}\|_\infty \leq 2\|\nu_n^\eta\|_\infty + \frac{3\|T^*\|_{2,\infty}G\sqrt{2\log(1/\gamma)}}{\alpha_nn^{1/2}}u_n\right) \\
	& \qquad - \Pr\left(\left\|\hat \nu_{n} - \nu_{n}\right\|_\infty + \frac{\|\hat T^*\|_{2,\infty}}{\alpha_n}\|R_{1n}\| + \|R_{2n}\|_\infty + \|R_{3n}\|_\infty > \frac{u_n}{2\alpha_nn^{1/2}\log n}\right) \\
	& \qquad - \Pr\left(\left\|\hat \nu_{n}^\eta - \nu_{n}^\eta\right\|_\infty + \frac{3G\sqrt{2\log(1/\gamma)}}{\alpha_nn^{1/2}}u_n\|\hat T^* - T^*\|_{2,\infty} > \frac{u_n}{2\alpha_nn^{1/2}\log n}\right) \\
	& \geq \Pr\left(\|\nu_{n}\|_\infty \leq 2\|\nu_n^\eta\|_\infty + \frac{3\|T^*\|_{2,\infty}G\sqrt{2\log(1/\gamma)}}{\alpha_nn^{1/2}}u_n\right) \\
	& \qquad - 2\sup_{P\in\Pc}\E_P\left[\left\|\hat \nu_{n} - \nu_{n}\right\|_\infty + \frac{\|\hat T^*\|_{2,\infty}}{\alpha_n}\|R_{1n}\| + \|R_{2n}\|_\infty + \|R_{3n}\|_\infty\right]u_n^{-1}\alpha_nn^{1/2}\log n \\
	& \qquad - 2\sup_{P\in\Pc}\E_P\left[\left\|\hat \nu_{n}^\eta - \nu_{n}^\eta\right\|_\infty u_n^{-1}\alpha_nn^{1/2}\log n + 3G\sqrt{2\log(1/\gamma)}\|\hat T^* - T^*\|_{2,\infty}\log n\right]. \\
	\end{aligned}
	\end{equation}
	Under Assumption~\ref{as:probability_rates_ci} (C1)-(C2) by Lemma~\ref{lemma:nu}
	\begin{equation*}
		\begin{aligned}
			\sup_{P\in\Pc}\E_P\left\|\hat \nu_{n} - \nu_{n}\right\|_\infty & = O\left(\frac{u_n\delta_{1n}}{\alpha_n^{3/2}n^{1/2}}\right), \\
			\sup_{P\in\Pc}\E_P\left\|\hat \nu_{n}^\eta - \nu_{n}^\eta\right\|_\infty & = O\left(\frac{u_n\delta_{1n}}{\alpha_n^{3/2}n^{1/2}} + \frac{u_n\delta_{2n}}{\alpha_nn^{1/2}}\right).
		\end{aligned}
	\end{equation*}
	Therefore, another application of Assumption~\ref{as:probability_rates_ci} and the order of $R_{2n}$ and $R_{3n}$ from the proof of Theorem~\ref{thm:bootstrap} imply
	\begin{equation*}
		\begin{aligned}
			\Pr\left(\varphi\in C_{n,1-\gamma}\right) & \geq \Pr\left(\|\nu_{n}\|_\infty \leq 2\|\nu_n^\eta\|_\infty + \frac{3\|T^*\|_{2,\infty}G\sqrt{2\log(1/\gamma)}}{\alpha_nn^{1/2}}u_n\right) \\
			& \qquad - O\left(\left(\frac{\delta_{3n}}{\alpha_n} + \frac{\delta_{1n}}{\alpha_n^{1/2}} + \alpha_n^{\beta\wedge 1}\right)u_n^{-1}\alpha_nn^{1/2}\log n + \delta_{2n}\log n \right).
		\end{aligned}
	\end{equation*}
	
	Next, $\|\nu_{n}\|_\infty$ is the supremum of empirical processes indexed by the class of functions
	\begin{equation*}
		\Hc_n = \left\{h:x\in C[0,1]^q\to u_n\left[(\alpha_nI + T^*T)^{-1}T^*x\right](z),\;z\in [0,1]^p\right\}.
	\end{equation*}
	
	For any $h\in\Hc_n$
	\begin{equation*}
		\|h(X_{ni})\|_\infty \leq u_n\|T^*\|_{2,\infty}\|(\alpha_nI + TT^*)^{-1}\|\|X_{ni}\| \leq \frac{\|T^*\|_{2,\infty}}{\alpha_n}Gu_n \triangleq H_n. 
	\end{equation*}
	So $H_n$ is the envelope for the class $\Hc_n$. Then by Proposition~\ref{thm:supremum_hoeffding}
	\begin{equation*}
		\begin{aligned}
			\Pr(\|\nu_n\|_\infty\leq \hat q_n) & = \Pr\left(\|\nu_n\|_\infty\leq 2\|\nu_n^\eta\|_\infty + 3H_n\sqrt{\frac{2\log(1/\gamma)}{n}} \right) \\
			& \geq 1 - \gamma,
		\end{aligned}
	\end{equation*}
	whence the result.
\end{proof}

The next result describes the convergence rate of the expected diameter of the confidence set based on the concentration inequality.
\begin{theorem}\label{thm:exp_diam_ci}
	Suppose that assumptions of Theorem~\ref{thm:main_ci} are satisfied. If $\delta_{1n}/\alpha_n^{1/2}=O(1)$, then
	\begin{equation*}
		\sup_{P\in\Pc}\E_P\left|C_{n,1-\gamma}\right|_\infty = O\left(\frac{u_n}{\alpha_n n^{1/2}}\right).
	\end{equation*}
\end{theorem}
\begin{proof}[Proof of Theorem~\ref{thm:exp_diam_ci}]
	Since $\E_P\left|C_{n,1-\gamma}\right|_\infty = 2\E_P \left|\hat q_n\right|$, under assumptions of Theorem~\ref{thm:main_ci}, there exists some universal constant $C$ such that
	\begin{equation*}
		\E_P\left|C_{n,1-\gamma}\right|_\infty \leq C\left\{\E_P\|\nu_n^\eta\|_\infty + \E_P\|\hat \nu_n^\eta - \nu_n^\eta\|_\infty + \frac{u_n}{\alpha_nn^{1/2}}\E_P\|\hat T^*\|_{2,\infty} \right\}.
	\end{equation*}
	By Jensen's inequality, since $\|X_{ni}\|\leq C$
	\begin{equation*}
		\begin{aligned}
			\E_P\|\nu_n^\eta\|_\infty & \leq \frac{\|T^*\|_{2,\infty}}{\alpha_n}\E_P\left\|\frac{u_n}{n}\sum_{i=1}^n\eta_iX_{ni}\right\| \\
			& \leq \frac{u_n\|T^*\|_{2,\infty}}{\alpha_nn^{1/2}}\sqrt{\E_P\|X_{ni}\|^2} \\
			& \leq \frac{u_n\|T^*\|_{2,\infty}}{\alpha_nn^{1/2}}G.
		\end{aligned}
	\end{equation*}
	Therefore by Lemma~\ref{lemma:nu} under Assumptions~\ref{as:source_condition1} and \ref{as:probability_rates_ci} (C1)-(C2)
	\begin{equation*}
		\sup_{P\in\Pc}\E_P\left|C_{n,1-\gamma}\right|_\infty = O\left(\frac{u_n}{\alpha_n n^{1/2}} + \frac{u_n\delta_{1n}}{\alpha_n^{3/2}n^{1/2}} + \frac{u_n\delta_{2n}}{\alpha_nn^{1/2}}\right).
	\end{equation*}
\end{proof}

\section{NPIV, functional regressions, and density deconvolution}
\subsection{NPIV model}
\begin{proposition}\label{prop:rate_of_density}
	Suppose that Assumption~\ref{as:data_npiv} (i)-(iii) is satisfied and the sequence of bandwidth parameters is such that $1/(nh_n^{p+q})=O(1)$, then for all $1<r<\infty$
	\begin{equation*}
		\left(\E\left\|\hat T^* - T^*\right\|_{2,\infty}^r\right)^{1/r} \leq  C\left\{\sqrt{\frac{\log h_n^{-1}}{nh_n^{p+q}}} + h_n^{s}\right\},
	\end{equation*}
	where the constant $C$ depends only on $s,M,r,\bar f,\|k\|_\infty$.
\end{proposition}
\begin{proof}
	By Lemma~\ref{lemma:operator_to_kernel},
	\begin{equation*}
		\begin{aligned}
			\E\left\|\hat T^* - T^*\right\|_{2,\infty}^r & = \E\left[\sup_{z\in(0,1)^p}\left(\int_{[0,1]^q}\left|\hat f_{ZW}(z,w) - f_{ZW}(z,w)\right|^2\dx w\right)^{1/2}\right]^r \\
			& \leq \E\left\|\hat f_{ZW} - f_{ZW}\right\|_\infty^{r}. \\
		\end{aligned}
	\end{equation*}
	By Minkowski inequality,
	\begin{equation*}
		\left(\E\left\|\hat f_{ZW} - f_{ZW}\right\|_\infty^{r}\right)^{1/r} \leq \left(\E\left\|\hat f_{ZW} - \E\hat f_{ZW}\right\|^r_\infty\right)^{1/r} + \left(\left\|\E\hat f_{ZW} - f_{ZW}\right\|_\infty^r\right)^{1/r}.
	\end{equation*}
	The order of the bias follows by standard computations under Assumption~\ref{as:data_npiv} (i) and (iii), see \cite{tsybakov2009introduction}
	\begin{equation*}
		\left\|\E\hat f_{ZW} - f_{ZW}\right\|_\infty = O\left(h_n^{s}\right),
	\end{equation*}
	where the big-O term is independent from the  distribution of $(Z,W)$. For the variance term, we apply the moment inequality in \cite{gine2015mathematical}, see Theorem 5.1.5 and Theorem 5.1.15, which gives
	\begin{equation*}
		\begin{aligned}
			\left(\E\left\|\hat f_{ZW} - \E\hat f_{ZW}\right\|^r_\infty\right)^{1/r} = O\left(\sqrt{\frac{\log h_n^{-1}}{nh_n^{p+q}}} + \sqrt{\frac{1}{nh_n^{p+q}}} + \frac{1}{nh_n^{p+q}}\right).
		\end{aligned}
	\end{equation*}
	Combining all estimates we obtain the result.
\end{proof}

\begin{proof}[Proof of Theorem~\ref{thm:cs_npiv}]
	We shall verify conditions of Theorems~\ref{thm:bootstrap} and \ref{thm:main_ci}. We show first that in the nonparametric IV model Assumption~\ref{as:probability_rates} (C3) is verified with $u_n=h_n^{-q/2}$ and $X_{ni}(w) = U_ih_n^{-q/2}K_w\left(h_n^{-1}(W_i-w)\right)$, where $[K_z\ast\varphi](Z) = h_n^{-p}\int\varphi(z)K_z\left(h_n^{-1}(Z-z)\right)\dx z$. Indeed, it is easy to see that
	\begin{equation*}
		\begin{aligned}
			(\hat r - \hat T\varphi)(w) & = \frac{h_n^{-q/2}}{n}\sum_{i=1}^nX_{ni}(w) + R_{1n}(w) \\
			R_{1n}(w) & = \frac{1}{nh_n^q}\sum_{i=1}^n\left\{\varphi(Z_i) - [K_z\ast\varphi](Z_i)\right\}K_w\left(h_n^{-1}(W_i - w)\right).
		\end{aligned}
	\end{equation*}
	Under Assumption~\ref{as:source_condition1}, $\varphi\in C_M^t[0,1]^p$, so that using standard bias computations, e.g. see \cite{tsybakov2009introduction}, under Assumption~\ref{as:data_npiv} (iii) there exists some constant $C<\infty$ such that
	\begin{equation*}
		\begin{aligned}
			\left\|\varphi - [K_z\ast \varphi]\right\|_\infty \leq  Ch_n^t,
		\end{aligned}
	\end{equation*}
	Then under Assumption~\ref{as:data_npiv} (i) and (ii)
	\begin{equation*}
		\begin{aligned}
			\sup_{P\in\Pc}\E_P\|R_{1n}\|^2 & \leq 
			\sup_{P\in\Pc}\E_P\left\|\frac{1}{n}\sum_{i=1}^n\left|h_n^{-q}K_w\left(h_n^{-1}(W_i - .)\right)\right|\right\|^2 Ch_n^t \\
			& \leq \sup_{P\in\Pc}\E_P\left\|h_n^{-q}K_w\left(h_n^{-1}(W_i - .)\right)\right\|^2Ch_n^t \\
			& = O(h_n^{t-q}).
		\end{aligned}
	\end{equation*}
	Therefore, Assumption~\ref{as:probability_rates} (C3) is verified with $\delta_{3n} = h_n^{(t-q)/2}$.
	
	Consider now the class of functions
	\begin{equation*}
		\Gc_{n} = \left\{(u,w)\mapsto u\frac{1}{h_n^{q/2}}K_w\left(h_n^{-1}(w-v)\right):\; v\in[0,1]^q\right\}.
	\end{equation*}
	Under Assumption~\ref{as:data_npiv} (iii) and (iv) this class admits a square-integrable envelope $G_{n}(u,w) = h_n^{-q/2}F\|K\|_\infty$ with
	\begin{equation*}
		\|G_{n}\|_{P,2} = h_n^{-q/2}\|K\|_\infty F
	\end{equation*}
	for all $P\in\Pc$. Recall that $K_w$ is a product kernel of functions of bounded variation under Assumption~\ref{as:data_npiv} (iii) and that the class of translations of functions of bounded variation is of the VC-type, so Assumption~\ref{as:vc_type} holds.
	
	Next, for any $g_n\in\Gc_n$
	\begin{equation*}
		\Var\left(g_n(v)\right) = \E\left|Uh_n^{-q/2}K_w\left(h_n^{-1}(W-v)\right)\right|^2,
	\end{equation*}
	so that by the law of iterated expectation and change of variables under Assumption~\ref{as:data_npiv} (ii) and (iv)
	\begin{equation*}
		\underline{\sigma}_2\|k\|^{2q}\underline{f} \leq \Var\left(g_n(v)\right) \leq F\|k\|^{2q}\bar f.
	\end{equation*}
	This verifies Assumption~\ref{as:variance_bounds} (i). Next, under Assumption~\ref{as:data_npiv} (iv)
	\begin{equation*}
		\begin{aligned}
			\|G\|_{P,4} & = h_n^{-q/2}\|k\|_\infty^q\|U\|_{P,4} \\
			& \leq Fh_n^{-q/2}\|k\|_\infty^q \\
			P|g_n|^3 & = h_n^{-3q/2}\E\left|UK_w\left(h_n^{-1}(W-v)\right)\right|^3 \\
			& \leq Fh_n^{-q/2}\|k\|^{3q}_3 \\
			P|g_n|^4 & = h_n^{-4q/2}\E\left|UK_w\left(h_n^{-1}(W-v)\right)\right|^4 \\
			& \leq Fh_n^{-q}\|k\|^{4q}_4.
		\end{aligned}
	\end{equation*}
	Therefore Assumption~\ref{as:variance_bounds} (ii) is satisfied with $b_n \leq Ch_n^{-q/2}$ and some universal constant $C<\infty$. Note that $b_n^4\log^7 n/n\to 0$ under Assumption~\ref{as:data_npiv} (vi) (c).
	
	It remains to verify Assumption~\ref{as:probability_rates} (C1) and (C2). The former follows from the Proposition~\ref{prop:rate_of_density} under Assumption~\ref{as:data_npiv} (i)-(iii) with $\delta_{1n} = \sqrt{\frac{\log h_n^{-1}}{nh_n^{p+q}}} + h_n^s$. Lastly, conditionally on the data $\Dc_n$
	\begin{equation*}
		\begin{aligned}
			\E_\varepsilon\left\|\hat V_n^\varepsilon - V_n^\varepsilon\right\|_\infty & = \E_\varepsilon\left\|\frac{1}{\sqrt{n}}\sum_{i=1}^n\varepsilon_i(\hat{\varphi}(Z_i) - \varphi(Z_i))h_n^{-q/2}K_w\left(h_n^{-1}(W_i-.)\right)\right\|_\infty \\
		\end{aligned}
	\end{equation*}
	is the expected value of the supremum of the Gaussian process indexed by the VC-type class. By Dudley-Sudakov's inequality, see \cite{dudley2016vn}, and change of variables
	\begin{equation*}
		\begin{aligned}
			\E_\varepsilon\left\|\hat V_n^\varepsilon - V_n^\varepsilon\right\|_\infty & \leq 24\int_0^{\|H_n\|_{P_n,2}}\sqrt{\log N(\Hc_n,\|.\|_{P_n,2},\epsilon)}\dx\epsilon \\
			& \leq 24\|H_n\|_{P_n,2}\int_0^1\sqrt{v\log\left(\frac{A}{\epsilon}\right)}\dx\epsilon,
		\end{aligned}
	\end{equation*}
	where
	\begin{equation*}
		\begin{aligned}
			\|H_n\|_{P_n,2}^2 & = \frac{1}{n}\sum_{i=1}^n(\hat\varphi(Z_i) - \varphi(Z_i))^2h_n^{-q}\|k\|_\infty^{2q}.
		\end{aligned}
	\end{equation*}
	Therefore
	\begin{equation*}
		\sup_{P\in\Pc}\E_P\left\|\hat V_n^\varepsilon - V_n^\varepsilon\right\|_\infty = \sup_{P\in\Pc}\E_P\left\|\hat\varphi - \varphi\right\|_\infty O\left(h_n^{-q/2}\right).
	\end{equation*}
	Under Assumptions~\ref{as:source_condition1} and \ref{as:data_npiv} (i)-(iv), it follows from \cite{babiiflorens2016b}, Theorem 4.2, that
	\begin{equation*}
		\sup_{P\in\Pc}\E_P\left\|\hat\varphi - \varphi\right\|_\infty = O\left(\frac{1}{\alpha_n}\left(\frac{1}{\sqrt{nh_n^q}} + h_n^{s}\right) + \frac{1}{\alpha_n^{1/2}}\sqrt{\frac{\log h_n^{-1}}{nh_n^{p+q}}} + \alpha_n^{\beta\wedge 1}\right).
	\end{equation*}
	
	Therefore, Assumption~\ref{as:probability_rates} (C2) is verified with
	\begin{equation*}
		\delta_{2n} = O\left(\frac{1}{\alpha_n}\left(\frac{1}{n^{1/2}h_n^q} + h_n^{s-q/2}\right) + \frac{1}{\alpha_n^{1/2}}\frac{\log^{1/2} h_n^{-1}}{n^{1/2}h_n^{p/2+q}} + \alpha_n^{\beta\wedge 1}h_n^{-q/2}\right).
	\end{equation*}
	Combining all estimates, it follows from Theorem~\ref{thm:bootstrap} that if $\frac{\log^7n}{nh_n^{2q}}\to 0$
	\begin{equation*}
		\inf_{P\in\Pc}\Pr(\varphi\in C_{n,1-\gamma}^*) \geq 1 - \gamma - O_p\left(\delta_n\right) - o_p(1)
	\end{equation*}
	with
	\begin{equation*}
		\delta_n = \left(\alpha_n^{-1/2}\delta_{1n} + \alpha_n^{-1}\delta_{3n} + \alpha_n^{\beta\wedge 1}\right)u_n^{-1}\alpha_nn^{1/2}\log n + \delta_{2n}\log n.
	\end{equation*}
	The result then follows under Assumption~\ref{as:data_npiv} (vi) after taking the probability limit.
	
	For confidence sets based on the concentration inequality, we note that Assumptions~\ref{as:probability_rates_ci} (C1) and (C3) hold with the same $\delta_{1n}$ and $\delta_{3n}$ sequences. Note also that $\|X_{ni}\|\leq F\|k\|^q\triangleq G$. It only remains to verify Assumption~\ref{as:probability_rates_ci} (C2). To that end note that
	\begin{equation*}
		\begin{aligned}
			\E_P\left\|\hat X_{ni} - X_{ni}\right\|^2 & = \E_P\left\|\frac{1}{h_n^{q/2}}(\hat\varphi(Z_i) - \varphi(Z_i))K_w(h_n^{-1}(W_i - .))\right\|^2 \\
			& \leq \E_P\|\hat\varphi - \varphi\|^2_\infty\|K_w\|^2.
		\end{aligned}
	\end{equation*}
	Inspection of the proof of the uniform risk bound in \cite{babiiflorens2016b} reveals that for all $r\geq 2$ we also have
	\begin{equation}\label{eq:uniform_risk_bound}
	\left(\sup_{P\in\Pc}\E_P\|\hat\varphi - \varphi\|_\infty^r\right)^{1/r} = O\left(\frac{1}{\alpha_n}\left(\frac{1}{\sqrt{nh_n^q}} + h_n^{s}\right) + \frac{1}{\alpha_n^{1/2}}\sqrt{\frac{\log h_n^{-1}}{nh_n^{p+q}}} + \alpha_n^{\beta\wedge 1}\right),
	\end{equation}
	provided that
	\begin{equation*}
		\sup_{P\in\Pc}\left(\E_P\left\|\hat T^* - T^*\right\|_{2,\infty}^{2r}\right)^{\frac{1}{2r}} = O\left(\sqrt{\frac{\log h_n^{-1}}{nh_n^{p+q}}}\right)
	\end{equation*}
	and
	\begin{equation}\label{eq:moment_r}
	\sup_{P\in\Pc}\left(\|\hat r - \E\hat r\|^{2r}\right)^{\frac{1}{2r}} = O\left(\frac{1}{\sqrt{nh_n^q}} + h_n^s\right).
	\end{equation}
	The former statement is established in Proposition~\ref{prop:rate_of_density}, while for the latter we use the following inequality valid for any i.i.d. sequence $(\xi_{ni})_{i=1}^n$
	\begin{equation*}
		\left(\E\left\|\frac{1}{n}\sum_{i=1}^n\xi_{ni}\right\|^{r}\right)^\frac{1}{r} \leq C\left\{n^{-1/2}\left(\E\|\xi_{ni}\|^2\right)^{1/2} + n^{1/r-1}\left(\E\|\xi_{ni}\|^r \right)^{1/r} \right\}.
	\end{equation*}
	In particular, setting $\xi_{ni} = Y_ih_n^{-q}K_w(h_n^{-1}(W_i - w)) - \E\left[Y_ih_n^{-q}K_w(h_n^{-1}(W_i - w))\right]$, under Assumption~\ref{as:data_npiv} (iv) we obtain Eq.~\ref{eq:moment_r} with $r=2$. Therefore, Assumption~\ref{as:probability_rates_ci} (C2) is verified with
	\begin{equation*}
		\delta_{2n} = \left(\frac{1}{\alpha_n}\left(\frac{1}{\sqrt{nh_n^q}} + h_n^s\right) + \frac{1}{\alpha_n^{1/2}}\sqrt{\frac{\log h_n^{-1}}{nh_n^{p+q}}} + \alpha_n^{\beta \wedge 1}\right)
	\end{equation*}
	and by Theorem~\ref{thm:main_ci}
	\begin{equation*}
		\inf_{P\in\Pc}\Pr\left(\varphi\in C_{n,1-\gamma}\right) \geq 1 - \gamma - O(\delta_n)
	\end{equation*}
	with
	\begin{equation*}
		\delta_n = \delta_{2n}\log n + \left(\frac{\delta_{1n}}{\alpha_n^{1/2}} + \frac{\delta_{3n}}{\alpha_n} + \alpha_n^{\beta\wedge 1}\right)u_n^{-1}\alpha_nn^{1/2}\log n.
	\end{equation*}
\end{proof}

\begin{proof}[Proof of Theorem~\ref{thm:npiv_cs_npiv}]
	For concentration-based confidence sets we apply Theorem~\ref{thm:exp_diam_ci}. In the proof of Theorem~\ref{thm:cs_npiv} we show that $\delta_{1n} = \sqrt{\frac{\log n}{nh_n^{p+q}}} + h_n^{s}$. Under Assumption~\ref{as:data_npiv} (vi) $\delta_{1n}/\alpha_n^{1/2}\to 0$, whence since $u_n = h_n^{-q/2}$, by Theorem~\ref{thm:exp_diam_ci}
	\begin{equation*}
		\sup_{P\in\Pc}\E_P\left|C_{n,1-\gamma}\right|_\infty = O\left(\frac{1}{\alpha_n\sqrt{nh_n^q}}\right).
	\end{equation*}
	Next we verify Assumption~\ref{as:bootstrap_cs_diameters}. Note that (i) holds, since the class is described by the convolution kernel of bounded variation. For (ii) we note that
	\begin{equation*}
		\E_P\left\|\hat X_{ni} - X_{ni}\right\|^2_\infty \leq \E_P\|\hat\varphi - \varphi\|^2_\infty\|K_w\|^2_\infty h_n^{-q},
	\end{equation*}
	which tends to zero due to Eq.~\ref{eq:uniform_risk_bound} and Assumption~\ref{as:data_npiv} (vi). Lastly, (iii) follows since $\E_P\|X_{ni}\|^2_\infty \leq F\|K_w\|_\infty^2h_n^{-q}$. Therefore, we obtain the second statement by Theorem~\ref{thm:exp_diam_bootstrap}.
\end{proof}

\subsection{Functional regressions}
\begin{proof}[Proof of Theorem~\ref{thm:confidence_set_flir}]
	We verify assumptions of Theorems~\ref{thm:bootstrap} and \ref{thm:main_ci}. Under Assumption~\ref{as:data_flir} (ii), by Hoffman-J\o rgensen's inequality, e.g., see \cite{gine2015mathematical}, Theorem 3.1.15,
	\begin{equation*}
		\begin{aligned}
			\left(\E_P\left\|\hat T^* - T^*\right\|_{2,\infty}^2\right)^{1/2} & \leq \left(\E_P\left\|\frac{1}{n}\sum_{i=1}^nZ_iW_i - \E [Z_iW_i]\right\|_{\infty}^2\right)^{1/2} \\
			& = 12\sqrt{3}\left(16\E_P\left\|\frac{1}{n}\sum_{i=1}^nZ_iW_i - \E [Z_iW_i]\right\|_{\infty} + \frac{2FC}{n}\right).\\
		\end{aligned}
	\end{equation*} 
	Next, under Assumption~\ref{as:data_flir} (iii), by the bracketing moment inequality, e.g., see \cite{gine2015mathematical}, Propositions 3.5.15 and 4.3.36,
	\begin{equation*}
		\sup_{P\in\Pc}\E_P\left\|\frac{1}{n}\sum_{i=1}^nZ_iW_i - \E[Z_iW_i]\right\|_\infty = O\left(\frac{1}{\sqrt{n}}\right).
	\end{equation*}
	Therefore, Assumptions~\ref{as:probability_rates} (C1) and \ref{as:probability_rates_ci} (C1) are satisfied with $\delta_{1n}=1/\sqrt{n}$. 
	
	Since
	\begin{equation*}
		\hat r - \hat T\varphi = \frac{1}{n}\sum_{i=1}^nU_iW_i,
	\end{equation*}
	Assumptions~\ref{as:probability_rates} (C3) and \ref{as:probability_rates_ci} (C3) are trivially satisfied with $X_{ni} = U_iW_i$, $u_n = 1$, and $\delta_{3n}=0$.
	
	Note also that by the Cauchy-Schwartz inequality
	\begin{equation*}
		\begin{aligned}
			\E_P\left\|\hat X_{ni} - X_{ni}\right\|^2 & = \E_P\left\|\frac{1}{n}\sum_{i=1}^n\langle\hat\varphi - \varphi,Z_i\rangle W_i\right\|^2 \\
			& \leq \E_P\|\hat\varphi - \varphi\|^2.
		\end{aligned}
	\end{equation*}
	Under Assumptions~\ref{as:data_flir} (i)-(ii), it follows from \cite{babiiflorens2016b} that
	\begin{equation}\label{eq:risk_bound_flir}
	\sup_{P\in\Pc}\E_P\|\hat\varphi - \varphi\|^2 = O\left(\frac{1}{\alpha_nn} + \alpha_n^{(2\beta+1)\wedge 2}\right),
	\end{equation}
	whence Assumption~\ref{as:probability_rates_ci} (C2) is verified with $\delta_{2n} = \frac{1}{\sqrt{\alpha_nn}} + \alpha_n^{(\beta+1/2)\wedge 1}$. Lastly, $\|X_{ni}\|\leq 1 \triangleq G$. Then by Theorem~\ref{thm:main_ci} for confidence sets based on the concentration inequality, we obtain
	\begin{equation*}
		\inf_{P\in\Pc}\mathrm{Pr}\left(\varphi\in C_{n,1-\gamma}\right) \geq 1 - \gamma - O\left(\left(\alpha_n^{\beta\wedge 1+1}n^{1/2} + \frac{1}{\sqrt{\alpha_nn}} + \alpha_n^{\beta/2\wedge 1}\right)\log n\right).
	\end{equation*}
	For bootstrap confidence sets we need additionally to verify Assumptions~\ref{as:vc_type}, \ref{as:variance_bounds}, and \ref{as:probability_rates} (C2). The relevant class of functions is $\Gc = \left\{g_s(u,w) = uw(s):\; s\in[0,1]^p \right\}$. 
	This class is of VC-type under Assumption~\ref{as:data_flir} (iii), e.g., see \cite{gine2015mathematical}, Proposition 3.6.12. Therefore Assumption~\ref{as:vc_type} is satisfied. Next under Assumption~\ref{as:data_flir} (iv)
	\begin{equation*}
		\Var(g_s) = \E|UW(s)|^2
	\end{equation*}
	is uniformly bounded away from zero and infinity, which verifies Assumption~\ref{as:variance_bounds} (i). Next, under Assumption~\ref{as:data_flir} (ii)
	\begin{equation*}
		\begin{aligned}
			\|G\|_{P,4} & = F \\
			P|g_s|^3 & = \E\left|UW(s)\right|^3 \leq F^3 \\
			P|g_s|^4 & = \E\left|UW(s)\right|^4 \leq F^4.
		\end{aligned}
	\end{equation*}
	Therefore Assumption~\ref{as:variance_bounds} (ii) is satisfied with universal constants $b,\sigma$. Therefore, $b^4\log^7n/n\to 0$.
	
	Lastly, note that for some universal constant $C$ depending only on $\|r\|_\infty$ and $K$ 
	\begin{equation*}
		\begin{aligned}
			\E_P\left\|\hat V_n^\varepsilon - V_n^\varepsilon \right\|_\infty & = \E_P\E_\varepsilon\left\|\frac{1}{\sqrt{n}}\sum_{i=1}^n\varepsilon_i\langle \hat\varphi - \varphi,Z_i\rangle W_i \right\|_\infty \\
			& \leq CF\E_P\|\hat\varphi - \varphi\|,
		\end{aligned}
	\end{equation*}
	where the inequality follows by Dudley-Sudakov's bound. Therefore, Assumption~\ref{as:probability_rates} (C2) is verified with the same $\delta_{2n} = \frac{1}{\sqrt{\alpha_nn}} + \alpha_n^{\beta/2\wedge 1}$. Then, by Theorem~\ref{thm:bootstrap}
	\begin{equation*}
		\inf_{P\in\Pc}\Pr(\varphi\in C_{n,1-\gamma}^*) \geq 1 - \gamma - O_p\left(\left(\alpha_n^{\beta\wedge 1+1}n^{1/2} + \frac{1}{\sqrt{\alpha_nn}} + \alpha_n^{\beta/2\wedge 1}\right)\log n\right) - o_p(1)
	\end{equation*}
	and the result follows by taking the probability limit.
\end{proof}
\begin{proof}[Proof of Theorem~\ref{thm:flir_cs_rate}]
	For concentration-based confidence sets we apply Theorem~\ref{thm:exp_diam_ci}. In the proof of Theorem~\ref{thm:confidence_set_flir} we show that $\delta_{1n}=n^{-1/2}$. Under Assumption~\ref{as:data_flir} (vi), $\delta_{1n}/\alpha_n^{1/2}\to 0$ and $\delta_{2n}\to 0$, whence since $u_n = 1$, by Theorem~\ref{thm:exp_diam_ci}
	\begin{equation*}
		\sup_{P\in\Pc}\E_P\left|C_{n,1-\gamma}\right|_\infty = O\left(\frac{1}{\alpha_nn^{1/2}}\right).
	\end{equation*}
	Next we verify Assumption~\ref{as:bootstrap_cs_diameters}. (i) holds, under Assumption~\ref{as:data_flir} (iii). For (ii), under Assumption~\ref{as:data_flir} (ii)
	\begin{equation*}
		\E_P\left\|\hat X_{ni} - X_{ni}\right\|^2_\infty \leq CF\E_P\|\hat\varphi - \varphi\|^2,
	\end{equation*}
	which tends to zero due to Eq.~\ref{eq:risk_bound_flir} and Assumption~\ref{as:data_flir} (vi). Lastly, (iii) follows trivially under Assumption~\ref{as:data_flir} (ii). Therefore, we obtain the second statement by Theorem~\ref{thm:exp_diam_bootstrap}.
\end{proof}

\subsection{Density deconvolution}
\begin{proof}[Proof of Theorem~\ref{thm:confidence_set_deconv}]
	We verify assumptions of Theorems~\ref{thm:bootstrap} and \ref{thm:main_ci}. Note that Assumptions~\ref{as:probability_rates} (C1) and \ref{as:probability_rates_ci} (C1) are trivially satisfied with $\delta_{1n}=0$. Next,
	\begin{equation*}
		\hat r - T\varphi = \frac{1}{nh_n}\sum_{i=1}^nK\left(\frac{Y_i - y}{h_n}\right) - \frac{1}{h_n}\E\left[K\left(\frac{Y_i - y}{h_n}\right)\right] + R_{1n}
	\end{equation*}
	with
	\begin{equation*}
		\begin{aligned}
			\left\|R_{1n}\right\|^2 & = \int_0^1\left|\frac{1}{h_n}\E\left[K\left(\frac{Y_i - y}{h_n}\right)\right] - r(y)\right|^2\dx y.
		\end{aligned}
	\end{equation*}
	Under Assumption~\ref{as:data_decon} (i)-(iv) by \cite{tsybakov2009introduction}, Proposition 1.5, we have
	\begin{equation*}
		\sup_{P\in\Pc}\E_P\|R_{1n}\|^2 = O(h_n^{2s}).
	\end{equation*}
	Therefore, Assumptions~\ref{as:probability_rates} (C3) and \ref{as:probability_rates_ci} (C3) are verified with $X_{ni} = \frac{1}{h_n^{1/2}}K\left(\frac{Y_i - y}{h_n}\right) - \frac{1}{h_n^{1/2}}\E\left[K\left(\frac{Y_i - y}{h_n}\right)\right]$, $u_n = h_n^{-1/2}$, and $\delta_{3n}=h_n^s$.
	
	Note also that
	\begin{equation*}
		\begin{aligned}
			\E_P\left\|\hat X_{ni} - X_{ni}\right\|^2 & = \E_P\left\|\frac{1}{nh_n^{1/2}}\sum_{i=1}^nK\left(\frac{Y_i - .}{h_n}\right) - \frac{1}{h_n^{1/2}}\E\left[K\left(\frac{Y_1 - .}{h_n}\right)\right] \right\|^2 \\
			& \leq \frac{\|K\|^2}{n}.
		\end{aligned}
	\end{equation*}
	Therefore, Assumption~\ref{as:probability_rates_ci} (C2) is verified with $\delta_{2n} = n^{-1/2}$. Lastly, putting $[K_h\ast r](y) = h_n^{-1}\int K(h_n^{-1}(y-u))r(u)\dx u$
	\begin{equation*}
		\begin{aligned}
			\|X_{ni}\|^2 & = \left\|h_n^{-1/2}K\left(\frac{Y_i-.}{h_n}\right) - h_n^{1/2}[K_h\ast r]\right\|^2 \\
			& \leq \frac{2}{h_n}\int K^2\left(\frac{Y_i - y}{h_n}\right)\dx y + 2h_n\|K_h\ast r\|^2 \\
			& \leq 4\|K\|^2 \triangleq G^2,
		\end{aligned}
	\end{equation*}
	where the last inequality follows by Young's inequality for convolutions. Then by Theorem~\ref{thm:main_ci} for confidence sets based on the concentration inequality, we obtain
	\begin{equation*}
		\inf_{P\in\Pc}\mathrm{Pr}\left(\varphi\in C_{n,1-\gamma}\right) \geq 1 - \gamma - O\left(\left(h_n^{s} + \alpha_n^{\beta\wedge 1+1}\right)\sqrt{nh_n}\log n + \frac{\log n}{n^{1/2}}\right).
	\end{equation*}
	For bootstrap confidence sets we also need to verify Assumptions~\ref{as:vc_type}, \ref{as:variance_bounds}, and \ref{as:probability_rates} (C2). The relevant class of functions is
	\begin{equation*}
		\Gc_n = \left\{\tilde y \mapsto \frac{1}{h_n^{1/2}}K\left(\frac{y - \tilde y}{h_n}\right):\; y\in[0,1] \right\}.
	\end{equation*}
	This class admits the envelope $G(y) = h_n^{-1/2}\|K\|_\infty$ and is of VC-type under Assumption~\ref{as:data_decon} (vi), e.g., see \cite{gine2015mathematical}, Proposition 3.6.12. Therefore Assumption~\ref{as:vc_type} is satisfied. For $g_n\in\Gc_n$
	\begin{equation*}
		\Var(g_n) = \E\left|\frac{1}{h_n^{1/2}}K\left(\frac{Y_1 - y}{h_n}\right)\right|^2 - \left(\E\left[\frac{1}{h_n^{1/2}}K\left(\frac{Y_1 - y}{h_n}\right)\right]\right)^2
	\end{equation*}
	so that for sufficiently large $n$
	\begin{equation*}
		0<\inf_{y\in[0,1]} r(y)\|K\|^2 - h_n\|K\|_1^2\|r\|_\infty^2 \leq \Var\left(g_n\right) \leq \|K\|^2\|r\|_\infty.
	\end{equation*}
	This verifies Assumption~\ref{as:variance_bounds} (i). Next, under Assumption~\ref{as:data_decon} (iv)
	\begin{equation*}
		\begin{aligned}
			\|G\|_{P,4} & = h_n^{-1/2}\|K\|_\infty \\
			P|g_n|^3 & = h_n^{-3/2}\E\left|K\left(h_n^{-1}(Y-y)\right)\right|^3 \\
			& \leq h_n^{-1/2}\|K\|^{3}_3\|r\|_\infty \\
			P|g_n|^4 & = h_n^{-4/2}\E\left|K\left(h_n^{-1}(Y-y)\right)\right|^4 \\
			& \leq h_n^{-1}\|K\|^{4}_4\|r\|_\infty.
		\end{aligned}
	\end{equation*}
	Therefore, Assumption~\ref{as:variance_bounds} (ii) is satisfied with $b_n = O(h_n^{-1/2})$ and some universal constant $\sigma$. Note also that $b_n^4\log^7n/n\to 0$ under Assumption~\ref{as:data_decon}
	
	Lastly, note that for some universal constant $C$ depending only on $\|r\|_\infty$ and $K$ 
	\begin{equation*}
		\begin{aligned}
			\E_P\left\|\hat V_n^\varepsilon - V_n^\varepsilon \right\|_\infty & = \E\left|\frac{1}{\sqrt{n}}\sum_{i=1}^n\varepsilon_i\right|\E_P\left\|\frac{1}{nh_n^{1/2}}\sum_{i=1}^nK\left(\frac{Y_i - .}{h_n}\right) - \frac{1}{h_n^{1/2}}\E\left[K\left(\frac{Y_1 - .}{h_n}\right)\right] \right\|_\infty \\
			& \leq C\sqrt{\frac{\log h_n^{-1}}{nh_n}},
		\end{aligned}
	\end{equation*}
	where the last line follows by \cite{gine2002rates}, Theorem 2.1, under Assumption~\ref{as:data_decon} (i), (ii), (iv), and (vi). This verifies Assumption~\ref{as:probability_rates} (C2) with $\delta_{2n} = \sqrt{\frac{\log h_n^{-1}}{nh_n}}$. Therefore, by Theorem~\ref{thm:bootstrap}
	\begin{equation*}
		\inf_{P\in\Pc}\Pr(\varphi\in C_{n,1-\gamma}^*) \geq 1 - \gamma - O_p\left(\left(h_n^s + \alpha_n^{\beta\wedge 1+1}\right)\sqrt{nh_n}\log n + \sqrt{\frac{\log h_n^{-1}}{nh_n}}\log n\right) - o_p(1)
	\end{equation*}
	and the result follows by taking the probability limit.
\end{proof}

\begin{proof}[Proof of Theorem~\ref{thm:deconv_cs_rate}]
	For concentration-based confidence sets we apply Theorem~\ref{thm:exp_diam_ci}. In the proof of Theorem~\ref{thm:confidence_set_flir} we show that $\delta_{1n}=0$. Therefore, by Theorem~\ref{thm:exp_diam_ci} with $u_n=h_n^{-1/2}$
	\begin{equation*}
		\sup_{P\in\Pc}\E_P\left|C_{n,1-\gamma}\right|_\infty = O\left(\frac{1}{\alpha_n\sqrt{nh_n^q}}\right)
	\end{equation*}
	Next we verify Assumption~\ref{as:bootstrap_cs_diameters}. Note that (i) holds, since the class is described by the translation of the kernel function of bounded variation. For (ii) we note that by Hoffman-J\o rgensen's inequality
	\begin{equation*}
		\begin{aligned}
			& \E_P\left\|\hat X_{ni} - X_{ni}\right\|^2_\infty \\
			& \leq \E_P\left\|\frac{1}{nh_n^{1/2}}\sum_{i=1}^nK\left(\frac{Y_i-y}{h_n}\right) - \frac{1}{h_n^{1/2}}\E\left[K\left(\frac{Y_i-y}{h_n}\right)\right]\right\|^2_\infty \\
			& \leq 432\left(\E_P\left\|\frac{1}{nh_n^{1/2}}\sum_{i=1}^nK\left(\frac{Y_i-y}{h_n}\right) - \frac{1}{h_n^{1/2}}\E\left[K\left(\frac{Y_i-y}{h_n}\right)\right]\right\|_\infty + \frac{2}{nh_n^{1/2}}\right)^2.
		\end{aligned}
	\end{equation*}
	Therefore,
	\begin{equation*}
		\sup_{P\in\Pc}\E_P\left\|\hat X_{ni} - X_{ni}\right\|^2_\infty = O\left(\frac{\log h_n^{-1}}{n} + h_n^{2s} + \frac{1}{n^2h_n}\right)
	\end{equation*}
	which tends to zero under Assumption~\ref{as:data_decon} (vi). Lastly, (iii) follows since $\E_P\|X_{ni}\|^2_\infty \leq 2\|K\|_\infty^2h_n^{-1}$. Therefore, we obtain the second statement by Theorem~\ref{thm:exp_diam_bootstrap}.
\end{proof}

\section{Technical lemmas}\label{app:online}
We first characterize the bias of the regularized estimator uniformly over our class of models.
\begin{proof}[Proof of Proposition~\ref{prop:regularization_bias}]
	By Lemma~\ref{lemma:infinity_to_hilbert}, $(\alpha_n I + T^*T)$ is an invertible operator between $(C,\|.\|_\infty)$ spaces. Using $f(T^*T)T^*=T^*f(TT^*)$ with $f(x)=(\alpha_n+x)^{-1}$, and factorizing the operator norm $\|T^*g(TT^*)\phi\|_\infty \leq \|T^*\|_{2,\infty}\|g(TT^*)\|\|\phi\|$ with $g(x)=\alpha_n(\alpha_n + x)^{-1}x^\beta$, under Assumption~\ref{as:source_condition1} for any $(\varphi,T)\in\Fc$ 
	\begin{equation}\label{eq:bias1}
	\begin{aligned}
	\left\|(\alpha_n I + T^*T)^{-1}T^*r - \varphi\right\|_\infty & = \left\|\left[(\alpha_n I + T^*T)^{-1}T^*T - I\right]\varphi\right\|_\infty \\
	& = \left\|\alpha_n (\alpha_n I + T^*T)^{-1}\varphi\right\|_\infty \\
	& = \left\|\alpha_n(\alpha_n I + T^*T)^{-1}(T^*T)^\beta T^*\psi\right\|_\infty \\
	& = \left\|T^*\alpha_n(\alpha_n I + TT^*)^{-1}(TT^*)^\beta\psi\right\|_\infty \\ 
	& \leq \|T^*\|_{2,\infty}\left\|\alpha_n(\alpha_n I + TT^*)^{-1}(TT^*)^\beta\right\|\|\psi\| \\
	\end{aligned}
	\end{equation}
	Note that for $b\in(0,1)$, the function $\lambda\mapsto \frac{\lambda^b}{\alpha_n + \lambda}$ is strictly concave on $(0,\infty)$ admitting its maximum at $\lambda = \frac{b}{1-b}\alpha_n$. On the other hand, for $b\in[1,\infty)$, this function is strictly increasing on $[0,\|T\|^2]$, reaching its maximum at the end of this interval. Therefore, by isometry of functional calculus
	\begin{equation*}
		\left\|\alpha_n(\alpha_n I + TT^*)^{-1}(TT^*)^\beta\right\| = \alpha_n\sup_{\lambda\in[0,\|T\|^2]}\left|\frac{\lambda^\beta}{\alpha_n + \lambda}\right| \leq \tilde R\alpha_n^{\beta\wedge 1}
	\end{equation*}
	with $\tilde R = \beta^\beta(1-\beta)^{1-\beta}\one_{0<\beta<1} + C^{2(\beta-1)}\one_{\beta\geq1}$. Therefore,
	\begin{equation*}
		\sup_{(\varphi,T)\in\Fc}\left\|(\alpha_n I + T^*T)^{-1}T^*r - \varphi\right\|_\infty \leq R\alpha_n^{\beta\wedge 1}
	\end{equation*}
	with $R = C^2\tilde R$.
\end{proof}

We also need the following inequality known in the theory of numerical ill-posed inverse problems.
\begin{lemma}\label{lemma:infinity_to_hilbert}
	Suppose that $T:L_2[a,b]^p\to L_2[a,b]^q$ is an integral operator with a continuous kernel function. Then for any $\alpha>0$ $(\alpha I + T^*T)$ is invertible as an operator from $\Rc(T^*T)\subset (C,\|.\|_\infty)$ to $(C,\|.\|_\infty)$ space and
	\begin{equation*}
		\left\|(\alpha I + T^*T)^{-1}\right\|_{\infty} \leq \frac{\|T^*\|_{2,\infty}/2 + \alpha^{1/2}}{\alpha^{3/2}}.
	\end{equation*}
\end{lemma}
\begin{proof}
	Continuity of the kernel function ensures that $T^*T$ maps to $C[a,b]^p$. Since for any $\alpha>0$,  $\alpha I + T^*T$ is invertible as an operator on the $L_2$ space, see \cite{nair2009linear}, Lemma 4.1, and $(C,\|.\|_\infty)\subset L_2$, it is also invertible as an operator between $(C,\|.\|_\infty)$.
	
	Note that for any $\varphi\in C[a,b]^p$
	\begin{equation*}
		\left[(\alpha_nI + T^*T)^{-1}T^*T - I\right]\varphi = -\alpha_n(\alpha_n I + T^*T)^{-1}\varphi.
	\end{equation*}
	Using this identity
	\begin{equation}\label{eq:rajan}
	\left\|(\alpha_n I + T^*T)^{-1}\right\|_{\infty} = \sup_{\|\varphi\|_\infty=1}\left\|(\alpha_n I + T^*T)^{-1}\varphi\right\|_\infty \leq \frac{\|(\alpha_n I + T^*T)^{-1}T^*T\|_\infty + 1}{\alpha_n}.
	\end{equation}
	Factoring the norm as $\|T^*\psi\|_\infty\leq \|T^*\|_{2,\infty}\|\psi\|$, and using the isometry of functional calculus
	\begin{equation*}
		\begin{aligned}
			\left\|(\alpha_n I + T^*T)^{-1}T^*T\right\|_\infty & = \left\|T^*(\alpha_n I + TT^*)^{-1}T\right\|_\infty \\
			& \leq \|T^*\|_{2,\infty}\sup_{\lambda\in[0,\|T\|^2]}\left|\frac{\lambda^{1/2}}{\alpha_n + \lambda}\right| \\
			& = \frac{\|T^*\|_{2,\infty}}{2\alpha_n^{1/2}}.
		\end{aligned}
	\end{equation*}
	Combining this with Eq.~\ref{eq:rajan} gives the result.
\end{proof}

For the next result, put
\begin{equation*}
	\xi_n = \left[(\alpha_n I + \hat T^*\hat T)^{-1}\hat T^*\hat T - (\alpha_n I + T^*T)^{-1}T^*T\right]\varphi.
\end{equation*}
\begin{lemma}\label{lemma:xi_term}
	Suppose that $T$ and $\hat T$ are integral operators in $\Lc_2$ with continuous kernel functions mapping from $L_2[a,b]^p$ to a subset of $L_2[a,b]^q$. Then under Assumption~\ref{as:source_condition1}, there exists a constant $C<\infty$ that does not depend on $(\varphi,T)$ such that
	\begin{equation*}
		\|\xi_n\|_\infty \leq \frac{C}{\alpha_n^{1/2}}\left(\|\hat T^* - T^*\|_{2,\infty} + \|\hat T^* - T^*\|_{2,\infty}^2\right).
	\end{equation*}
\end{lemma}
\begin{proof}
	Decompose
	\begin{equation*}
		\begin{aligned}
			\xi_n & = \left[(\alpha_n I + \hat T^*\hat T)^{-1}\hat T^*\hat T - (\alpha_n I + T^*T)^{-1}T^*T\right]\varphi  \\
			& = \left[\alpha_n(\alpha_n I + T^*T)^{-1} - \alpha_n(\alpha_n I + \hat T^*\hat T)^{-1}\right]\varphi \\
			& = (\alpha_nI + \hat T^*\hat T)^{-1}\left[\hat T^*\hat T - T^*T\right]\alpha_n(\alpha_nI + T^*T)^{-1}\varphi \\
			& \leq (\alpha_n I + \hat T^*\hat T)^{-1}\hat T^*(\hat T - T)\alpha_n(\alpha_n I + T^*T)^{-1}\varphi \\
			& + (\alpha_n I + \hat T^*\hat T)^{-1}(\hat T^* - T^*)\alpha_nT(\alpha_n I + T^*T)^{-1}\varphi \\
			& \equiv I_n + II_n.
		\end{aligned}
	\end{equation*}
	Both terms involve the error from the estimation of the operators and regularization bias. For the first term, we factor the operator norm $\|\hat T^*\psi\|_\infty \leq \|\hat T^*\|_{2,\infty}\|\psi\|$
	\begin{equation*}
		\begin{aligned}
			\|I_n\|_\infty & \leq \|\hat T^*\|_{2,\infty}\left\|(\alpha_n I + \hat T\hat T^*)^{-1}\right\|\|\hat T - T\|\left\|\alpha_n(\alpha_n I + T^*T)^{-1}\varphi\right\|
		\end{aligned}
	\end{equation*}
	By isometry of functional calculus
	\begin{equation*}
		\left\|(\alpha_n I + \hat T^*\hat T)^{-1}\right\| =  \sup_{\lambda\in[0,\|\hat T\|^2]}\left|\frac{1}{\alpha_n + \lambda}\right| \leq \frac{1}{\alpha_n},\qquad \mathrm{a.s.}
	\end{equation*}
	Under Assumption~\ref{as:source_condition1}, the last term is
	\begin{equation*}
		\left\|\alpha_n(\alpha_n I + T^*T)^{-1}(T^*T)^\beta T^*\psi\right\| \leq \alpha_n\sup_{\lambda\in[0,\|T\|^2]}\left|\frac{\lambda^{\beta+1/2}}{\alpha_n + \lambda}\right|C = O\left(\alpha_n^{(\beta+1/2)\wedge 1}\right).
	\end{equation*}
	Combining these findings with the following bound
	\begin{equation*}
		\|\hat T - T\|=\|\hat T^* - T^*\|\leq (b-a)^{p/2}\|\hat T^* - T^*\|_{2,\infty},
	\end{equation*}
	we have for all $(\varphi,T)\in\Fc$
	\begin{equation*}
		\|I_n\|_\infty = O\left(\frac{\alpha_n^{\beta\wedge 1/2}}{\alpha_n^{1/2}}\|\hat T^*-T^*\|_{2,\infty}\|\hat T^*\|_{2,\infty}\right).
	\end{equation*}
	
	For the second term, we factor the operator norm in the following way:
	\begin{equation*}
		\|II_n\|_\infty \leq \left\|(\alpha_n I + \hat T^*\hat T)^{-1}\right\|_\infty\left\|\hat T^* - T^*\right\|_{2,\infty}\left\|\alpha_nT(\alpha_n I + T^*T)^{-1}\varphi\right\|
	\end{equation*}
	By Lemma~\ref{lemma:infinity_to_hilbert}
	\begin{equation*}
		\left\|(\alpha_n I + \hat T^*\hat T)^{-1}\right\|_\infty \leq \frac{\|\hat T^*\|_{2,\infty}/2 + \alpha_n^{1/2}}{\alpha_n^{3/2}}
	\end{equation*}
	Under Assumption~\ref{as:source_condition1},
	\begin{equation*}
		\left\|\alpha_nT(\alpha_n I + T^*T)^{-1}(T^*T)^\beta T^*\psi\right\| \leq\alpha_n\sup_{\lambda\in[0,\|T\|^2]}\left|\frac{\lambda^{\beta+1}}{\alpha_n + \lambda}\right| C = O(\alpha_n).
	\end{equation*}
	Combining all above findings, we have uniformly in $(\varphi, T)$
	\begin{equation*}
		\|II_n\|_\infty = O\left(\|\hat T^* - T^*\|_{2,\infty}\left(\frac{\|\hat T^*\|_{2,\infty}}{\alpha_n^{1/2}} + 1\right)\right),
	\end{equation*}
	and the conclusion follows from collecting all estimates and by triangle inequality.
\end{proof}

For the following lemma denote
\begin{equation*}
	\begin{aligned}
		\nu_n & = \frac{u_n}{n}\sum_{i=1}^n(\alpha_n I + T^*T)^{-1}T^*X_{ni}, \\
		\hat\nu_n & = \frac{u_n}{n}\sum_{i=1}^n(\alpha_n I + \hat T^*\hat T)^{-1}\hat T^*X_{ni}, \\
		\nu_n^\eta & = \frac{u_n}{n}\sum_{i=1}^n(\alpha_n I + T^*T)^{-1}T^*\eta_iX_{ni}, \\
		\hat\nu_n^\eta & = \frac{u_n}{n}\sum_{i=1}^n(\alpha_n I + \hat T^*\hat T)^{-1}\hat T^*\eta_i\hat X_{ni}. \\
	\end{aligned}
\end{equation*}

\begin{lemma}\label{lemma:nu}
	For any $T,\hat T\in\Lc_2$ we have
	\begin{equation*}
		\begin{aligned}
			\left\|\hat\nu_{n} - \nu_{n}\right\|_\infty & \leq \left(\frac{1}{\alpha_n} + \frac{1}{2\alpha_n^{3/2}}\|T^*\|_{2,\infty} \right)\|\hat T^* - T^*\|_{2,\infty}\left\|\frac{u_n}{n}\sum_{i=1}^nX_{ni}\right\|, \\
			\left\|\hat\nu_{n}^\eta - \nu_{n}^\eta\right\|_\infty & \leq \left(\frac{1}{\alpha_n} + \frac{1}{2\alpha_n^{3/2}}\|T^*\|_{2,\infty} \right)\|\hat T^* - T^*\|_{2,\infty}\left\|\frac{u_n}{n}\sum_{i=1}^n\eta_i X_{ni}\right\| \\
			& \qquad + \frac{\|\hat T^*\|_{2,\infty}}{\alpha_n}\left\|\frac{u_n}{n}\sum_{i=1}^n\eta_i(\hat X_{ni} - X_{ni})\right\|. \\
		\end{aligned}
	\end{equation*}
\end{lemma}
\begin{proof}
	Decompose
	\begin{equation*}
		\begin{aligned}
			& \hat\nu_n - \nu_n = \\
			& = \left[\hat T^*(\alpha_n I + \hat T\hat T^*)^{-1} - T^*(\alpha_n I + TT^*)^{-1}\right]\frac{u_n}{n}\sum_{i=1}^nX_{ni} \\
			& = \left[(\hat T^* - T^*)(\alpha_n I + \hat T\hat T^*)^{-1} - T^*(\alpha_n I + \hat T\hat T^*)^{-1}(\hat T\hat T^* - TT^*)(\alpha_n I + TT^*)^{-1}\right]\frac{u_n}{n}\sum_{i=1}^nX_{ni} \\
			& = \left[(\hat T^* - T^*)(\alpha_n I + \hat T\hat T^*)^{-1} - T^*(\alpha_n I + \hat T\hat T^*)^{-1}\hat T(\hat T^* - T^*)(\alpha_n I + TT^*)^{-1}\right]\frac{u_n}{n}\sum_{i=1}^nX_{ni} \\
			& \quad - T^*(\alpha_n I + \hat T\hat T^*)^{-1}(\hat T - T)T^*(\alpha_n I + TT^*)^{-1}\frac{u_n}{n}\sum_{i=1}^nX_{ni}.
		\end{aligned}
	\end{equation*}
	Note that
	\begin{equation*}
		\|\hat T - T\| = \|\hat T^* - T^*\| \leq \|\hat T^* - T^*\|_{2,\infty}
	\end{equation*}
	and
	\begin{equation*}
		\|(\alpha_n I + \hat T\hat T^*)^{-1}\hat T\|\leq \frac{1}{2\alpha_n^{1/2}},\qquad \|T^*(\alpha_n I + TT^*)^{-1}\| \leq \frac{1}{2\alpha_n^{1/2}}.
	\end{equation*}
	Then
	\begin{equation*}
		\begin{aligned}
			\|\hat\nu_n - \nu_n\|_\infty & \leq \left(\frac{1}{\alpha_n} + \frac{1}{2\alpha_n^{3/2}}\|T^*\|_{2,\infty} \right)\|\hat T^* - T^*\|_{2,\infty}\left\|\frac{u_n}{n}\sum_{i=1}^nX_{ni}\right\|. \\
		\end{aligned}
	\end{equation*}
	For the second statement note that
	\begin{equation*}
		\begin{aligned}
			\hat\nu_n^\eta - \nu_n^\eta & = \left[\hat T^*(\alpha_n I + \hat T\hat T^*)^{-1} - T^*(\alpha_n I + TT^*)^{-1}\right]\frac{u_n}{n}\sum_{i=1}^n\eta_iX_{ni} \\
			& \qquad + (\alpha_nI + \hat T^*\hat T)^{-1}\hat T^*\frac{u_n}{n}\sum_{i=1}^n\eta_i(\hat X_{ni} - X_{ni}) \\
			& \triangleq I_n + II_n.
		\end{aligned}
	\end{equation*}
	Note that the bound on $I_n$ follows trivially from above computations if we replace $X_{ni}$ by $\eta_iX_{ni}$, while
	\begin{equation*}
		\|II_n\|_\infty \leq \frac{\|\hat T^*\|_{2,\infty}}{\alpha_n}\left\|\frac{u_n}{n}\sum_{i=1}^n\eta_i(\hat X_{ni} - X_{ni})\right\|.
	\end{equation*}
\end{proof}

\begin{lemma}\label{lemma:operator_to_kernel}
	Suppose that $T^*\in\Lc_{2,\infty}$ is an integral operator with continuous kernel function $k$, then
	\begin{equation*}
		\|T^*\|_{2,\infty} = \sup_{z\in[a,b]^p}\left(\int_{[a,b]^q}|k(z,w)|^2\dx w\right)^{1/2}\equiv \|k\|_{2,\infty},
	\end{equation*}
	where $\|k\|_{2,\infty}$ is a mixed norm on the iterated space $L_\infty[a,b]^p(L_2[a,b]^q)$.
\end{lemma}
\begin{proof}
	By Cauchy-Schwartz inequality, we have $\|T^*\|_{2,\infty}\leq \|k\|_{2,\infty}$. On the other side, continuity of $k$ implies that $\exists z_0\in[a,b]^p$ such that $\|k\|_{2,\infty} = \left(\int_{[a,b]^q}|k(z_0,w)|^2\dx w\right)^{1/2}$. Take $\psi(w) = \frac{k(z_0,w)}{\|k(z_0,.)\|}$. Then $\|\psi\|=1$, and
	\begin{equation*}
		\|T^*\|_{2,\infty} \geq \|T^*\psi\|_{\infty} \geq |(T^*\psi)(z_0)| = \|k(z_0,.)\| = \|k\|_{2,\infty}.
	\end{equation*}
\end{proof}

\end{document}